\documentclass{amsart}

\usepackage{xypic}
\input xy
\xyoption{all}
\usepackage{epsfig}
\usepackage{amsthm}
\usepackage{amssymb, times}
\usepackage{amsmath}
\usepackage{amscd}
\usepackage{color}

\newcommand{\lb}{\varLambda}

\newcommand{\bg}{\begin{equation}}
\newcommand{\ed}{\end{equation}}
\newcommand{\bga}{\begin{eqnarray}}
\newcommand{\eda}{\end{eqnarray}}

\def\cbdu{\par{\raggedleft$\Box$\par}}

\newtheorem {Theorem}  {Theorem}

\numberwithin{Theorem}{section}

\newtheorem {Lemma}[Theorem]  {Lemma}

\theoremstyle{definition}
\newtheorem{Definition}[Theorem]{Definition}
\theoremstyle{remark}

\expandafter\chardef\csname pre amssym.def
at\endcsname=\the\catcode`\@ \catcode`\@=11
\def\undefine#1{\let#1\undefined}
\def\newsymbol#1#2#3#4#5{\let\next@\relax
 \ifnum#2=\@ne\let\next@\msafam@\else
 \ifnum#2=\tw@\let\next@\msbfam@\fi\fi
 \mathchardef#1="#3\next@#4#5}
\def\mathhexbox@#1#2#3{\relax
 \ifmmode\mathpalette{}{\m@th\mathchar"#1#2#3}%
 \else\leavevmode\hbox{$\m@th\mathchar"#1#2#3$}\fi}
\def\hexnumber@#1{\ifcase#1 0\or 1\or 2\or 3\or 4\or 5\or 6\or 7\or 8\or
 9\or A\or B\or C\or D\or E\or F\fi}

\font\teneufm=eufm10 \font\seveneufm=eufm7 \font\fiveeufm=eufm5
\newfam\eufmfam
\textfont\eufmfam=\teneufm \scriptfont\eufmfam=\seveneufm
\scriptscriptfont\eufmfam=\fiveeufm

\catcode`\@=\csname pre amssym.def at\endcsname

\newcounter{remark}
\newcommand{\e}{\epsilon}

\renewcommand{\th}{\theta}

\newcommand{\R}{\mathbf{R}}

\renewcommand{\div}{\mbox{div}}

\def  \R   {{\mathbb R}}

\def  \T   {{\mathbb T}}

\def  \12  {{\frac{1}{2}}}
\def  \p   {\partial}

\def\build#1_#2^#3{\mathrel{\mathop{\kern 0pt#1}\limits_{#2}^{#3}}}

\begin{document}

\title[Determining modes for SQG]{Determining modes for the surface quasi-geostrophic equation}


\author [Alexey Cheskidov]{Alexey Cheskidov}
\address{Department of Mathematics, Stat. and Comp.Sci.,  University of Illinois Chicago, Chicago, IL 60607,USA}
\email{acheskid@uic.edu} 
\author [Mimi Dai]{Mimi Dai}
\address{Department of Mathematics, Stat. and Comp.Sci.,  University of Illinois Chicago, Chicago, IL 60607,USA}
\email{mdai@uic.edu} 

\thanks{The authors were partially supported by NSF grants
DMS--1108864, DMS--1517583, and DMS--1815069.}





\begin{abstract}
We introduce a determining wavenumber for the surface quasi-geostrophic (SQG) equation 
defined for each individual trajectory and then study its dependence on the
force. While in the subcritical and critical cases this wavenumber has a uniform upper bound, it may blow up
when the equation is supercritical. A bound on the determining wavenumber provides determining modes,
and measures the number of degrees of freedom of the flow, or resolution needed to describe a solution to the SQG equation.
\bigskip

{\em This paper is dedicated to Professor Edriss S. Titi on the occasion of his sixtieth birthday with friendship and admiration.}

\bigskip


\hspace{0.02cm}CLASSIFICATION CODE: 35Q35, 37L30.
\end{abstract}

\maketitle


\section{Introduction}
\label{sec-intro}

In this paper we introduce a determining wavenumber $\lb_\theta(t)$ for the forced surface quasi-geostrophic (SQG) equation
\begin{equation}\begin{split}\label{QG}
\frac{\partial\theta}{\partial t}+u\cdot\nabla \theta+\nu\Lambda^\alpha\theta =f,\\
u=R^\perp\theta,
\end{split}
\end{equation}
on the torus $\mathbb T^2=[0,L]^2$, where $0<\alpha<2$, $\nu>0$, $\Lambda=\sqrt{-\Delta}$ is the Zygmund operator, and
\bg\notag
R^\perp\theta=\Lambda^{-1}(-\partial_2\theta,\partial_1\theta).
\ed
The scalar function $\theta$ represents the potential temperature and the vector function $u$ represents the fluid velocity. 
The initial data $\th(0) \in L^2(\mathbb{T}^2)$ and the force $f \in L^p(\mathbb{T}^2)$ for some $p>2/\alpha$ are assumed to have zero average.

The wavenumber $\lb_\theta(t)$ is defined solely based on the structure of the equation, but not on the force, regularity properties, or any known bounds on the solution.
We prove that if two complete weak solutions $\theta_1, \theta_2 \in L^\infty((-\infty,\infty);L^2)$  (i.e., lying on the global attractor) coincide on frequencies below
$\max\{\lb_{\theta_1}, \lb_{\theta_2}\}$, then $\theta_1 \equiv \theta_2$. While in the subcritical and critical cases this wavenumber has uniform upper bounds, it may blow up when the equation is supercritical. A bound on $\lb_\theta$ immediately
provides determining modes, which in some sense measure the number of degrees of freedom of the flow, or resolution needed to describe a solution to the SQG equation.

The first result of finite dimensionality of a flow was obtained by Foias and Prodi for the 2D Navier-Stokes equations (NSE) in \cite{FP}, where it was shown that low modes control high modes asymptotically as time goes to infinity. Then an explicit estimate on the number of determining modes was obtained by Foias, Manley, Temam, and Treve in \cite{FMTT}, and improved by Jones and Titi in \cite{JT}.
A related result, the finite dimensionality of the global attractor of the 2D NSE, was first proved by Foias and Temam in \cite{FT} (see Constantin, Foias, and Temam
\cite{CFT} for the best available bound). 
See also \cite{CFMT, FJKT, FMRT, FT84, FTiti} and references therein for more results in this direction.

Equation (\ref{QG}) with $\alpha=1$ describes the evolution of the surface temperature field in a rapidly rotating and stably stratified fluid with potential velocity \cite{CMT}. Being applicable in atmosphere and oceanography,
this model  is also very interesting from the mathematical point of view. Indeed, the behavior of solutions to (\ref{QG}) with $\nu=0$ in 2D and the behavior of potentially singular solutions to the Euler's equation in 3D have been found to be similar  both analytically and numerically (see \cite{CCW, CMT, CW, Pe} and the references therein).
Since $L^\infty$, the highest controlled norm, is critical when $\alpha=1$, 
equation (\ref{QG}) is referred as supercritical, critical and subcritical SQG for $0<\alpha<1$, $\alpha=1$ and $\alpha>1$ respectively.
The global regularity problem of the critical SQG equation is very challenging due to the balance of the nonlinear term and the dissipative term
in (\ref{QG}). This problem has since been resolved, with several different proofs and their adaptations to the case of a smooth force available \cite{CaV, CTV, CVicol, FPV, KN09, KNV}.

The long time behavior of  solutions to the critical SQG equations have been studied in 
\cite{CD,CCV, CTVan, CTV, Don, SS03, SS05}.
The first result on the existence of an attractor was obtained recently by Constantin, Tarfulea, and Vicol in  \cite{CTV}, where the authors studied the long time dynamics of regular solutions of the forced critical SQG equations using the nonlinear maximal principle \cite{CVicol}.
With the assumption that the force $f\in L^\infty(\mathbb T^2)\cap H^1(\mathbb T^2)$ and the initial data in $H^1(\mathbb T^2)$, the authors proved the existence of a compact attractor, which is a global attractor in the
classical sense in $H^s$ for $s\in(1,3/2)$, and it attracts all the points (but not bounded sets) in $H^1$. 
Moreover, the authors proved that the attractor has a finite box-counting dimension.

Later, Cheskidov and Dai \cite{CD} proved that the critical SQG equation (\ref{QG}) with $\alpha=1$ possesses a global attractor in
$L^2(\mathbb T^2)$, provided the force $f$ is solely in $L^p$ for $p>2$. As the first step, it is established that for any initial data in $L^2$, a weak (viscosity) solution is bounded in $L^\infty$ on any interval $[t_0, \infty)$, $t_0>0$. The main tool is an application of the De Giorgi iteration method to the forced critical SQG as it was done by Caffarelli and Vasseur in \cite{CaV} in the unforced case. This is the only 
part that requires the force to be in $L^p$ for some $p>2$. Second, in the spirit of  \cite{CCFS}, the Littlewood-Paley decomposition technique is used to show that bounded weak solutions have zero energy flux and hence satisfy the energy equality. The energy equality immediately implies the continuity of weak solutions in $L^2$. In the third step, an abstract framework of evolutionary systems introduced by Cheskidov and Foias \cite{CF} was followed to show the existence of a weak
global attractor. Finally, with all the above ingredients at hand, an abstract result established by Cheskidov in \cite{C} was applied to prove that the
weak global attractor is in fact a strongly compact strong global attractor. 

In a very recent paper \cite{CCV}, Constantin, Coti Zelati, and Vicol  showed that the $H^1$ attractor obtained in \cite{CTV} is indeed a global attractor in the classical sense,
i.e., it attracts bounded sets
in $H^1$. The main ingredient here is an estimate of a $C^\alpha$ norm of a solution in terms of the $L^\infty$ norms of the solution and the force, which was done using 
the Constantin-Vicol nonlinear maximal principle \cite{CVicol}. Since the $L^\infty$ norm is known to be bounded
thanks to the De Giorgi iteration  method, this automatically gives an absorbing ball in $C^\alpha$, which in turn implies the existence
of absorbing balls in $H^1$ and $H^{3/2}$, and hence asymptotic compactness in $H^1$. This results in the existence of the $H^1$ global attractor.

In this paper we start by introducing  a time-dependent determining wavenumber $\lb_\theta(t)$ defined for each individual trajectory $\theta(t)$ and then study its dependence on  $\alpha$ and $f$. Given a weak solution $\theta(t)$  of the SQG equation, we define 
\[
\lb_{\theta,r}(t)=\min\{\lambda_q:\lambda_p^{1-\alpha+\frac2r}\|\theta_{p}\|_r< c_{\alpha,r}\nu \quad \forall p>q, \quad \mbox { and }
\quad \lambda_q^{-\alpha} \sum_{p\leq q}\lambda_p\|\theta_p\|_\infty< c_{\alpha,r}\nu   \},
\]
for $r \in \mathcal{I}_\alpha$. Note that we use a convention $\min{\emptyset}=\infty$. Here $\lambda_q =\frac{2^q}{L}$, $\theta_q = \Delta_q \theta$ is the Littlewood-Paley projection of $\theta$ (see Section~\ref{sec:pre}), and
\[
c_{\alpha,r}=\left\{
\begin{split}
&\frac{c_0}{\alpha^2(r+1)^2}\left(1-2^{\frac2{r+1}-\frac2r}\right)^{\frac{\alpha(r+1)}{2}},&\quad 0<\alpha\leq 1,\\
&c_0(\alpha-1)^2\left(1-2^{\frac{\alpha-1}2-\frac2r}\right)^{\frac{2\alpha}{\alpha-1}}, &\quad  1<\alpha< 2,
\end{split}
\right.
\]
\[
\mathcal{I}_\alpha =  
\left\{
\begin{split}
&\textstyle \left(\frac{4}{\alpha}-1, \infty\right),&\quad 0<\alpha\leq 1,\\
&\textstyle\left(\frac{2\alpha}{\alpha-1}, \frac{4}{\alpha-1}\right),&\quad  1<\alpha< 2.
\end{split}
\right.
\]
for some absolute adimensional constant $c_0$ (hence $c_{\alpha,r}$ is adimensional). Actually, the unit for $c_0$ is $[c_0] = [\theta]/[u]$, but the SQG equation \eqref{QG} is written so that $\theta$ and $u$ have the same units. 

The first part in the definition of $\lb$ resembles the dissipation wavenumber introduced
by Cheskidov and Shvydkoy in \cite{CS} for the 3D Navier-Stokes equation, also defined in terms of a critical norm, but $L^\infty$ based, i.e., the smallest one. In \cite{CS} it was shown that in some sense the linear term is dominant
above that wavenumber. More precisely,  it is enough to control a weak solution of the 3D Navier-Stokes equations in the inertial range, i.e., below the dissipation wavenumber, in order to ensure regularity. The dissipation wavenumber was also adapted to the supercritical SQG by Dai in \cite{D},
where the smallest critical norm was used as well.

The determining wavenumber is much
more restrictive than the dissipation wavenumber. First, a larger critical norm appears in the first
condition of the definition of $\lb_{\theta,r}$. Second, $\lb_{\theta,r}$ not only controls high modes, but also low modes, as can be seen in the second condition. From the mathematical point of view, this is due to the fact that
there are more terms to control and fewer cancellations in this setting.

In the first part of the paper we show that $\lb$ is indeed a determining wavenumber. 
\begin{Theorem}\label{thm}
Let $\alpha \in (0,2)$ and $\theta_1(t)$ and $\theta_2(t)$ be weak solutions of the SQG equation~\eqref{QG}.  Let $\lb(t)=\max\{\lb_{\theta_1,r}(t), \lb_{\theta_2,r}(t)\}$
for some $r \in \mathcal{I}_\alpha$.
If
\begin{equation} \label{eq:dm-condition}
\theta_1(t)_{\leq \lb(t)}=\theta_2(t)_{\leq \lb(t)}, \qquad \forall t>0,
\end{equation}
then
\begin{equation}\notag
\lim_{t \to \infty} \|\theta_1(t) - \theta_2(t)\|_{B^0_{l,l}}=0,
\end{equation}
where $l=\alpha(r+1)/2$ when $\alpha \in (0,1]$, and $l=2\alpha/(\alpha-1)$ when $\alpha \in (1,2)$.

Moreover, if $\theta_1(t)$ and $\theta_2(t)$ are two complete (ancient) bounded in $L^2$ viscosity solutions,
i.e., $\theta_1, \theta_2 \in L^\infty((-\infty,\infty);L^2)$, and
\begin{equation} \label{eq:dm-condition}
\theta_1(t)_{\leq \lb(t)}=\theta_2(t)_{\leq \lb(t)}, \qquad \forall t<T,
\end{equation}
for some $T\in(-\infty,\infty]$, then
\[
\theta_1(t) = \theta_2(t), \qquad \forall t \leq T.
\]   
\end{Theorem}
Note that the second part of the theorem implies that for any solutions $\theta_1(t)$, $\theta_2(t)$ on
the attractor $\mathcal{A}$, we have $\theta_1 \equiv \theta_2$ provided $(\theta_1)_{\leq \lb} \equiv
(\theta_2)_{\leq \lb}$, where
\[
\mathcal{A} = \{\theta(0): \theta(t) \text{ is a complete bounded solution, i.e., } \theta \in L^\infty((-\infty,\infty);L^2)\}.
\]
In \cite{CD}, Cheskidov and Dai proved that $\mathcal{A}$ is a compact global attractor in the classical sense when $\alpha=1$. It uniformly attracts bounded sets in $L^2$,
it is the minimal closed attracting set, and it is the $L^2$-omega limit of the absorbing ball $B_{L^2}$. 
Clearly, this holds in the subcritical case $\alpha>1$ as well where we also have all the ingredients to apply
the framework of evolutionary systems \cite{C}. However, in the supercritical case $\alpha<1$, we only
know the existence of a weak global attractor at this point.

In the second part of the paper, using the De Giorgi iteration method, we extend the $L^\infty$ estimate in \cite{CD} to the whole range $\alpha >0$.
This argument requires the force $f$ to be in $L^p$ for some $p>2/\alpha$ and implies
\begin{equation} \label{eq:Linfty-est-intro}
\|\theta\|_\infty \lesssim \frac{\|f\|_p}{\nu}, \qquad \forall \theta \in \mathcal{A}.
\end{equation}
Note that this estimate holds for all $\alpha>0$, and it explains the choice of the space $B^0_{l,l}$ in Theorem~\ref{thm}. Indeed, thanks to the Littlewood-Paley Theorem or
simply the interpolation
\[
\|\theta\|_{B^0_{l,l}} \lesssim \|\theta\|_\infty^{1-\frac{2}{l}} \|\theta\|_2^{\frac{2}{l}},
\]
the $B^0_{l,l}$ norm of a viscosity solution is bounded on the global attractor. Thus, one can take a limit as the initial time goes to $-\infty$ and show that the difference between
two solutions that coincide below $\lb$ is zero (See Theorem \ref{thm-main-bound} and Section \ref{sec-pf-thm}). On the other hand, the $B^0_{l,l}$ norm enjoys a better estimate than the $L^l$ norm. Based on \eqref{eq:Linfty-est-intro}, we are able to establish an upper bound for $\lb_{\theta,r}$
in the subcritical case $\alpha \in (1,2)$. Namely, 

\begin{Theorem}\label{thm-theta-est1}
Assume $\alpha \in (1,2)$. For some $r\in\left(\frac{2\alpha}{\alpha-1}, \frac4{\alpha-1}\right)$, we have
\[
\lb_{\theta,r}\lesssim \left( \frac{\|f\|_2}{(\alpha-1)^2\nu^2} \right)^{\frac{2}{\alpha-1}},
\]
for large enough $t$ (or when $\theta \in \mathcal{A}$).  
\end{Theorem}

Here we took $p=2$ for simplicity. This estimate of $\lb_{\theta,r}$ gives the following bound on the number of determining modes $N$:
\[
N \lesssim  \left( \frac{\|f\|_2}{(\alpha-1)^2\nu^2} \right)^{\frac{4}{\alpha-1}}.
\]
In the critical case $\alpha=1$, the $L^\infty$ estimate
clearly is not enough to obtain a bound on $\lb$. However, combining it with the H\"older estimate
$\|\theta(t)\|_{C^h} \lesssim \|\theta(0)\|_{\infty} + \frac{\|f\|_\infty}{\nu}$ for some small $h$ and
large $t$, obtained by Constantin, Coti Zelati and Vicol in \cite{CCV}, we show that 
\begin{Theorem}\label{thm-theta-est2}
Let $\alpha=1$. Assume $f\in L^\infty\cap H^1$ and the initial data $\theta(0)\in H^1$. Then the estimate
\[
\lb_{\theta,r}\lesssim \left(\frac{\|f\|_\infty}{\nu^2}\right)^\frac{c\|f\|_\infty}{\nu^2},
\]
holds for some constant $c$ depending on $L$ and some large enough $r$.  
\end{Theorem}

Finally, the framework developed in this paper can be applied to other dissipative systems, such as the Navier-Stokes equations (see \cite{CD1,CDK}).

\bigskip

\section{Preliminaries}
\label{sec:pre}

\subsection{Notations}
We denote by $A\lesssim B$ an estimate of the form $A\leq C B$ with
some absolute constant $C$, and by $A\sim B$ an estimate of the form $C_1
B\leq A\leq C_2 B$ with some absolute constants $C_1$, $C_2$. 
We write $\|\cdot\|_p=\|\cdot\|_{L^p}$, and $(\cdot, \cdot)$ stands for the $L^2$-inner product.

\subsection{Littlewood-Paley decomposition}
\label{sec:LPD}
The techniques presented in this paper rely strongly on the Littlewood-Paley decomposition, which we recall  here briefly. For a more detailed description on  this theory we refer readers to the books by Bahouri, Chemin and Danchin \cite{BCD}, and Grafakos \cite{Gr}. 

Denote $\lambda_q=\frac{2^q}{L}$ for integers $q$. A nonnegative radial function $\chi\in C_0^\infty(\R^2)$ is chosen such that 
\begin{equation}\label{eq-xi}
\chi(\xi)=
\begin{cases}
1, \ \ \mbox { for } |\xi|\leq\frac{3}{4}\\
0, \ \ \mbox { for } |\xi|\geq 1.
\end{cases}
\end{equation}
Let 
\bg\notag
\varphi(\xi)=\chi(\xi/2)-\chi(\xi)
\ed
and
\begin{equation}\notag
\varphi_q(\xi)=
\begin{cases}
\varphi(2^{-q}\xi)  \ \ \ \mbox { for } q\geq 0,\\
\chi(\xi) \ \ \ \mbox { for } q=-1.
\end{cases}
\end{equation}
For a tempered distribution vector field $u$ we define its  Littlewood-Paley projection $u_q$ in the following way:
\begin{equation}\notag
\begin{split}
& h_q=\sum_{k\in \mathbb Z^2}\varphi_q(k)e^{i\frac{2\pi k\cdot x}{L}},\\
&u_q:=\Delta_qu=\sum_{k\in \mathbb Z^2}\hat u_k\varphi_q(k)e^{i\frac{2\pi k\cdot x}{L}}=\frac{1}{L^2}\int_{\T^2} h_q(y)u(x-y)dy,  \qquad q\geq -1,
\end{split}
\end{equation}
where $\hat{u}_k$ is the $k$th Fourier coefficient of $u$, i.e.,
\[
\hat{u}_k = \frac{1}{L^2}\int_{\mathbb{T}^2} u(x) e^{-i\frac{2\pi k \cdot x}{L}} \, dx.
\] 
Then we have
\bg\notag
u=\sum_{q=-1}^\infty u_q
\ed
in the sense of distributions. 
Let $h(x)$ be the inverse Fourier transform of $\phi(\xi)$, i.e.,
\[
h(x) = \int_{\mathbb{R}^2} \phi(\xi) e^{2\pi i \xi \cdot x} \, d\xi.
\]
By the Poisson summation formula, we have
\[
\begin{split}
\sum_{n\in \mathbb{Z}^2} h(x+2^qn)&=2^{-2q}\sum_{k\in \mathbb{Z}^2} \phi(2^{-q}k)e^{i\frac{2\pi 2^{-q}k \cdot x}{L}}\\
&=2^{-2q}h_q(2^{-q}x).
\end{split}
\]
In particular, 
\begin{equation} \label{eq:qwer}
\int_{\mathbb{T}^2}  |h_q(x)| \, dx  \lesssim 1, \qquad \int_{\mathbb{T}^2}  d(0,x)|\nabla h_q(x)| \, dx \lesssim 1,
\end{equation}
where $d(\cdot,\cdot)$ denotes the distance on $\mathbb{T}^2$.

To simplify the notation, we denote
\bg\notag
u_{\leq Q}=\sum_{q=-1}^Qu_q, \qquad u_{(Q,R]}=\sum_{p=Q+1}^Ru_p, \qquad \tilde u_q=\sum_{q-1\leq p\leq q+1}u_p. 
\ed
Notice that the identity 
\begin{equation}\label{near-sum}
\sum_{q-r\leq p\leq q+r}\Delta_qu_p=u_q, \qquad r=1,2,\dots
\end{equation}
holds since $\sum_{q-r\leq p\leq q+r}\varphi_q\varphi_p=\varphi_q$. In particular, $\Delta_q \tilde{u}_q = u_q$.

We will also use the Besov $B^0_{\ell,\ell}$ norm defined as
\[
\|u\|_{B^0_{\ell,\ell}} = \left( \sum_{q=-1}^\infty \|u_q\|_{\ell}^{\ell} \right)^\frac{1}{\ell}.
\]
The following inequalities will be used throughout the paper:
\begin{Lemma}\label{le:bern}(Bernstein's inequality) 
Let $r\geq s\geq 1$. Then for all tempered distributions $u$, 
\bg\notag
\|u_q\|_{r}\lesssim \lambda_q^{2(\frac{1}{s}-\frac{1}{r})}\|u_q\|_{s}.
\ed
\end{Lemma}

\begin{Lemma}\label{lp}
Assume $2\leq \ell<\infty$ and $0\leq \alpha\leq 2$. Then
\begin{equation}\notag
\ell\int u_q\Lambda^\alpha u_q|u_q|^{\ell-2} \, dx\gtrsim  \lambda_q^\alpha\|u_q\|_{\ell}^\ell.
\end{equation}
\end{Lemma}
For a proof of Lemma \ref{lp}, see \cite{CMZ, CC}.
\subsection{Riesz transform}
Here we recall the Riesz transform, a Fourier multiplier with the following symbol:
\[
(R_j f\hat)_k =
\left\{
 \begin{split}
 i \frac{k_j}{|k|} \hat{f}_k&, & k \ne 0,\\
 0&,  &k=0,
 \end{split}
 \right.
\]
where $j=1,2$.
Similar to \eqref{eq:qwer}, we can obtain  
\[
\|r_{j,q}\|_{L^1(\mathbb{T}^2)} \lesssim 1,
\] 
where
\[
r_{j,q}(x)= \sum_{k\in \mathbb Z^2\setminus \{0\}}i\frac{k_j}{|k|}\varphi_q(k)e^{i\frac{2\pi k\cdot x}{L}}.
\]
Hence, by Young's convolution inequality, $\|R_j \Delta_q f\|_\ell \lesssim \|f\|_\ell$. Moreover,
\[
\begin{split}
\|R_j f_q\|_\ell &= \|R_j \Delta_q \tilde{f}_q\|_\ell\\
& \leq  \|R_j \Delta_{q-1} f_q\|_\ell+\|R_j \Delta_q f_q\|_\ell+\|R_j \Delta_{q+1} f_q\|_\ell\\
& \lesssim\|f_q\|_\ell,
\end{split}
\]
for any $1\leq \ell\leq \infty$. Therefore,
\begin{equation} \label{eq:Riesz}
\|R^\perp f_q\|_\ell =\|(-R_2,R_1)f_q\|_\ell \lesssim \|f_q\|_\ell, \qquad 1\leq \ell \leq \infty.
\end{equation}

\subsection{Bony's paraproduct and commutator}
\label{sec-para}

Bony's paraproduct formula will be used to decompose the nonlinear term. First, note that
\[
u\cdot\nabla v =   \sum_{p}u_{\leq{p-2}}\cdot\nabla v_p + \sum_{p}u_{p}\cdot\nabla v_{\leq{p-2}} + \sum_{p}\sum_{|p-p'|\leq 1}u_p\cdot\nabla v_{p'}.
\]
Due to \eqref{eq-xi} we have $\varphi(\xi)=0$ when $|\xi|\leq 3/4$ or $|\xi|\geq 2$, and hence 
\[
(f_q  g_{\leq q-2})_{\geq q+2}=0, \qquad (f_q  g_{\leq q-2})_{\leq q-3}=0, \qquad (f_qg_{q+1})_{\geq q+3} =0,
\]
for tempered distributions $f$ and $g$. Therefore,
\begin{equation}\notag
\begin{split}
\Delta_q(u\cdot\nabla v)=&\sum_{q-1\leq p \leq q+2}\Delta_q(u_{\leq{p-2}}\cdot\nabla v_p)+
\sum_{q-1\leq p \leq q+2}\Delta_q(u_{p}\cdot\nabla v_{\leq{p-2}})\\
&+\sum_{p\geq q-2}\sum_{\substack{|p-p'|\leq 1\\ p' \geq q-2}}\Delta_q(u_p\cdot\nabla v_{p'}).
\end{split}
\end{equation}
It is usually sufficient to use a weaker form of this formula:
\begin{equation}\notag
\begin{split}
\Delta_q(u\cdot\nabla v)=&\sum_{|q-p|\leq 2}\Delta_q(u_{\leq{p-2}}\cdot\nabla v_p)+
\sum_{|q-p|\leq 2}\Delta_q(u_{p}\cdot\nabla v_{\leq{p-2}})\\
&+\sum_{p\geq q-2} \Delta_q(\tilde u_p \cdot\nabla v_p).
\end{split}
\end{equation}


We will also use the following  commutator notation
\begin{equation} \label{eq:CommutatodDef}
[\Delta_q, u_{\leq{p-2}}\cdot\nabla]v_p:=\Delta_q(u_{\leq{p-2}}\cdot\nabla v_p)-u_{\leq{p-2}}\cdot\nabla \Delta_qv_p.
\end{equation}
By definition of $\Delta_q$, for divergence free $u$ we have 
\begin{equation}\notag
\begin{split}
\left|[\Delta_q,u_{\leq{p-2}}\cdot\nabla] v_p\right|=&\left|\int_{\T^2}h_q(x-y)\left(u_{\leq p-2}(y)-u_{\leq p-2}(x)\right)\cdot \nabla v_p(y)\,dy\right|\\
=&\left|\int_{\T^2} \left(u_{\leq p-2}(y)-u_{\leq p-2}(x)\right)\cdot \nabla h_q(x-y) v_p(y)\,dy\right|\\
=&\left|\int_{\T^2}\frac{u_{\leq p-2}(y)-u_{\leq p-2}(x)}{d(y,x)}\cdot \nabla h_q(x-y)d(y,x) v_p(y)\,dy\right|,
\end{split}
\end{equation}
where we used integration by parts and the fact that $\div\, u_{\leq p-2}=0$. As before, $d(\cdot,\cdot)$ denotes the distance on $\mathbb{T}^2$.
Thus, by Young's inequality and \eqref{eq:qwer},
\begin{equation}\label{commu}
\begin{split}
\|[\Delta_q,u_{\leq{p-2}}\cdot\nabla] v_p\|_{\ell}
&\lesssim \|\nabla u_{\leq p-2}\|_\infty\|v_p\|_{\ell}\left|\int_{\T^2}d(0,x)|\nabla h_q(x)|\, dx\right|\\
&\lesssim \|\nabla u_{\leq p-2}\|_\infty\|v_p\|_{\ell},
\end{split}
\end{equation}
 for any $\ell>1$.

\bigskip

\section{Proof of the first part of Theorem \ref{thm}}
\label{sec:pf}



Now we are ready to prove our first main result, which holds for all weak solutions of the SQG equation, even the ones that might not satisfy the energy inequality.

\begin{Definition} \label{def:weaksolutions}
A weak solution to \eqref{QG} is a function $\th \in C_{\mathrm{w}}([0,T];L^2(\mathbb{T}^2))$ with zero spatial average that satisfies \eqref{QG} in a distributional sense.
That is, for any $\phi\in C_0^\infty(\mathbb T^2\times [0,T))$,
\begin{equation}\notag
-\int_0^T(\theta, \phi_t)dt-\int_0^T(u\theta, \nabla\phi)dt+\nu\int_0^T(\Lambda^{\frac{\alpha}{2}}\theta, \Lambda^{\frac{\alpha}{2}}\phi)dt
=(\theta_0, \phi(x,0))+\int_0^T(f, \phi)dt.
\end{equation}
\end{Definition}

Note that weak solutions satisfy this equality for weakly Lipschitz in time functions $\phi$ as well (see \cite{RRS}). This allows us to use Littlewood-Paley projections of $\theta$ as test functions.

\begin{Theorem}\label{thm-main-bound}
Let $\alpha \in (0,2)$, and $\theta_1(t)$, $\theta_2(t)$ be weak solutions of the SQG equation~\eqref{QG}. Let $\lb(t)=\max\{\lb_{\theta_1,r}(t), \lb_{\theta_2,r}(t)\}$
for some $r \in \mathcal{I}_\alpha$. Let
\begin{equation} \label{eq:defofl}
\ell=\left\{
\begin{split}
&\frac{\alpha(r+1)}{2}, \qquad 0<\alpha\leq 1,\\
&\frac{2\alpha}{\alpha-1}, \qquad  1<\alpha< 2.
\end{split}
\right.
\end{equation}
If
\begin{equation} \label{eq:dm-condition}
\theta_1(t)_{\leq \lb(t)}=\theta_2(t)_{\leq \lb(t)}, \qquad \forall t\in(T_1,T_2),
\end{equation}
then
\begin{equation} \label{eq:boundonldif-expon}
\|\theta_1(t) - \theta_2(t)\|_{B^0_{\ell,\ell}}^\ell \leq \|\theta_1(t_0) - \theta_2(t_0)\|_{B^0_{\ell,\ell}}^\ell e^{-\frac {c\nu}{L^\alpha} (t-t_0)}, \quad \forall t_0, t\in [T_1, T_2], \quad t_0 \leq t,
\end{equation}
where $c$ is an absolute constant.
\end{Theorem}

\begin{proof}
Denote $u_1=R^\perp\theta_1$ and $u_2=R^\perp\theta_2$.
Let $w=\theta_1-\theta_2$, which satisfies the equation
\begin{equation} \label{eq-w}
w_t+u_1\cdot\nabla w+\nu\Lambda^\alpha w+R^\perp w\cdot\nabla \theta_2=0
\end{equation}
in the sense of distributions. By our assumption, $w_{\leq \lb(t)} = 0$ for $t\in(T_1,T_2)$.
Recall that
\begin{equation}\label{para-r}
\mathcal{I}_\alpha =  
\left\{
\begin{split}
&\textstyle \left(\frac{4}{\alpha}-1, \infty\right),&\quad 0<\alpha\leq 1,\\
&\textstyle\left(\frac{2\alpha}{\alpha-1}, \frac{4}{\alpha-1}\right),&\quad  1<\alpha< 2.
\end{split}
\right.
\end{equation}
Combining $r \in \mathcal{I}_\alpha$ with \eqref{eq:defofl} one can verify that the conditions
\begin{equation}\notag
2 \leq \ell \leq r<\frac{2\ell}\alpha,\qquad
\frac2r+\frac\alpha \ell>\alpha-1, \qquad 1+\frac 2r-\frac\alpha \ell>0
\end{equation}
are satisfied. These inequalities will be used throughout the proof.

First, we project equation (\ref{eq-w}) onto the $q$-th shell, multiply it by $\ell w_q|w_q|^{\ell-2}$, and integrate, which is equivalent to using $\Delta_q(\ell w_q|w_q|^{\ell-2})$ as
a test function (see the remark after Definition~\eqref{def:weaksolutions}).   Adding up for all $q\geq -1$ and applying Lemma \ref{lp}, yields
\begin{equation}\label{w2}
\begin{split}
\|w(t)\|_{B^0_{\ell,\ell}}^\ell - \|w(t_0)\|_{B^0_{\ell,\ell}}^\ell + C \nu\int_{t_0}^t \|\Lambda^{\alpha/\ell} w\|_{B^0_{\ell,\ell}}^\ell\, d\tau \leq &\\
 -\int_{t_0}^t\ell\sum_{q\geq -1}\int_{\T^2}\Delta_q(R^\perp w\cdot\nabla \theta_2) & w_q|w_q|^{\ell-2}\, dx\, d\tau\\
-\int_{t_0}^t \ell\sum_{q\geq -1}\int_{\T^2}\Delta_q(u_1 \cdot\nabla w) & w_q|w_q|^{\ell-2}\, dx \, d\tau\\
=&\int_{t_0}^t I \, d\tau+ \int_{t_0}^t J \, d\tau,
\end{split}
\end{equation}
for all $T_1 \leq t_0 \leq t \leq T_2$.
Using Bony's paraproduct mentioned in Subsection \ref{sec-para}, $I$ is decomposed as
\begin{equation}\notag
\begin{split}
I=&
-\ell\sum_{q\geq -1}\sum_{|q-p|\leq 2}\int_{\T^2}\Delta_q(R^\perp w_{\leq{p-2}}\cdot\nabla(\theta_2)_p) w_q|w_q|^{\ell-2}\, dx\\
&-\ell\sum_{q\geq -1}\sum_{|q-p|\leq 2}\int_{\T^2}\Delta_q(R^\perp w_{p}\cdot\nabla (\theta_2)_{\leq{p-2}})w_q|w_q|^{\ell-2}\, dx\\
&-\ell\sum_{q\geq -1}\sum_{p\geq q-2}\int_{\T^2}\Delta_q(R^\perp\tilde w_p\cdot\nabla (\theta_2)_p)w_q|w_q|^{\ell-2}\, dx\\
 =& I_{1}+I_{2}+I_{3}.
\end{split}
\end{equation}
These terms are estimated as follows. First, recall $w_{\leq \lb(t)}=0$. Let $Q(t)$ be such that $\lb(t)=2^{Q(t)}/L$. Since $r \geq \ell$, we can choose $m$ so that 
$\frac1r+\frac1m+\frac{\ell-1}{\ell}=1$. Changing the order of summations and using H\"older's inequality, we infer
\begin{equation}\notag
\begin{split}
|I_{1}|&\leq \ell\sum_{q> Q}\sum_{\substack{|q-p|\leq 2 \\ p>Q+2}}\int_{\T^2}\left|\Delta_q(R^\perp w_{\leq{p-2}}\cdot\nabla(\theta_2)_p) w_q\right||w_q|^{\ell-2}\, dx\\
&= \ell\sum_{p> Q+2}\sum_{\substack{|q-p|\leq 2\\q>Q}}\int_{\T^2}\left|\Delta_q(R^\perp w_{\leq{p-2}}\cdot\nabla(\theta_2)_p) w_q\right||w_q|^{\ell-2}\, dx\\
&\lesssim \ell\sum_{p> Q+2}\lambda_p\|(\theta_2)_p\|_r\sum_{\substack{|q-p|\leq 2\\q>Q}}\| w_q\|_{\ell}^{\ell-1}\sum_{p'\leq p-2}\|R^\perp w_{p'}\|_m.\\
\end{split}
\end{equation}
 Then using the definition of $\lb_{\theta,r}$ and 
Young's inequality, we obtain 
\begin{equation}\notag
\begin{split}
|I_{1}|
&\lesssim c_{\alpha,r}\nu \ell\sum_{p> Q+2}\lambda_p^{\alpha-\frac 2r}\sum_{\substack{|q-p|\leq 2\\q>Q}}\| w_q\|_{\ell}^{\ell-1}\sum_{p'\leq p-2}\lambda_{p'}^{\frac2{\ell}-\frac2m}\|R^\perp w_{p'}\|_\ell\\
&\lesssim c_{\alpha,r}\nu \ell\sum_{p> Q}\lambda_p^{\alpha-\frac 2r}\| w_p\|_{\ell}^{\ell-1}\sum_{p'\leq p-2}\lambda_{p'}^{\frac2{\ell}-\frac2m}\|R^\perp w_{p'}\|_\ell\\
&\lesssim c_{\alpha,r}\nu \ell\sum_{p> Q}\lambda_p^{\frac{\alpha(\ell-1)}{\ell}}\| w_p\|_{\ell}^{\ell-1}\sum_{p'\leq p-2}\lambda_{p'}^{\frac \alpha \ell}\|R^\perp w_{p'}\|_{\ell}\lambda_{p-p'}^{\frac \alpha {\ell}-\frac2r}\\
&\lesssim c_{\alpha,r}\nu \ell\sum_{p> Q}\lambda_p^\alpha\|w_p\|_{\ell}^\ell+c_{\alpha,r}\nu \ell\sum_{p> Q}\left(\sum_{p'\leq p-2}\lambda_{p'}^{\frac \alpha {\ell}}\|R^\perp w_{p'}\|_{\ell}\lambda_{p-p'}^{\frac \alpha {\ell}-\frac2r}\right)^\ell.
\end{split}
\end{equation}
It then follows from Jensen's (or H\"older's) inequality, \eqref{eq:Riesz}, and the fact that $r<(2\ell)/\alpha$, 
\begin{equation}
\begin{split}
|I_{1}|
&\lesssim c_{\alpha,r}\nu \ell\sum_{p> Q}\lambda_p^\alpha\|w_p\|_{\ell}^\ell+c_{\alpha,r}\nu \ell\sum_{p'>Q}\lambda_{p'}^{\alpha}\|R^\perp w_{p'}\|_{\ell}^\ell\left(\sum_{p\geq p'-2}\lambda_{p-p'}^{\frac \alpha {\ell}-\frac2r}\right)^l\\
&\lesssim c_{\alpha,r}\nu \ell\sum_{p> Q}\lambda_p^\alpha\|w_p\|_{\ell}^\ell+c_{\alpha,r}\nu \ell\left(1-2^{\frac\alpha {\ell}-\frac2r}\right)^{-l}\sum_{q>Q}\lambda_q^\alpha\|w_q\|_{\ell}^\ell\\
&\lesssim c_{\alpha,r}\nu \ell\left(1-2^{\frac\alpha {\ell}-\frac2r}\right)^{-l}\sum_{q>Q}\lambda_q^\alpha\|w_q\|_{\ell}^\ell,\\
\end{split}
\end{equation}
where we needed $r<(2\ell)/\alpha$ in order to apply Jensen's inequality at the first step. 
For $I_2$ we first change the order of summations and decompose it into two parts:
\[
\begin{split}
|I_{2}|\leq &\ell\sum_{q>Q}\sum_{\substack{|q-p|\leq 2\\ p>Q}}\int_{\T^2}\left|\Delta_q(R^\perp w_{p}\cdot\nabla (\theta_2)_{\leq{p-2}})w_q\right| |w_q|^{\ell-2}\, dx\\
\leq &\ell\sum_{Q<p\leq Q+2}\sum_{\substack{|q-p|\leq 2\\ q>Q}}\int_{\T^2}\left|\Delta_q(R^\perp w_p\cdot\nabla (\theta_2)_{\leq{p-2}})w_q\right| |w_q|^{\ell-2}\, dx\\
&+\ell\sum_{p>Q+2}\sum_{\substack{|q-p|\leq 2\\ q>Q}}\int_{\T^2}\left|\Delta_q(R^\perp w_p\cdot\nabla (\theta_2)_{\leq{Q}})w_q\right| |w_q|^{\ell-2}\, dx\\
&+ \ell\sum_{p> Q+2}\sum_{\substack{|q-p|\leq 2\\q>Q}}\int_{\T^2}\left|\Delta_q(R^\perp w_{p}\cdot\nabla (\theta_2)_{(Q,p-2]})w_q\right| |w_q|^{\ell-2}\, dx\\
=& I_{21}+I_{21}'+I_{22}.
\end{split}
\]
Notice that $p-2\leq Q+2-2=Q$ in $I_{21}$, which implies that $I_{21}$ and $I_{21}'$ share the same estimate.
Using H\"older's inequality, the definition of $\lb_{\theta,r}$, and  Young's inequality for the first term, we obtain 
\[
\begin{split}
I_{21}'&\lesssim \ell\sum_{p> Q}\sum_{\substack{|p-q|\leq2\\q>Q}}\|\nabla(\theta_2)_{\leq Q}\|_\infty\|R^\perp w_p\|_{\ell}\|w_q\|_{\ell}^{\ell-1}\\
&\lesssim  c_{\alpha,r} \nu \ell\sum_{p>Q}\lambda_Q^\alpha \|R^\perp w_p\|_{\ell}\sum_{\substack{|p-q|\leq2\\q>Q}}\|w_q\|_{\ell}^{\ell-1}\\
&\lesssim c_{\alpha,r} \nu \ell\sum_{q>Q}\lambda_q^\alpha \| w_q\|_{\ell}^\ell.
\end{split}
\]
To estimate $I_{22}$, we first use H\"older's inequality, change the order of summations, use Bernstein's inequality 
\[
\begin{split}
I_{22}&\lesssim \ell\sum_{p> Q+2}\sum_{\substack{|p-q|\leq2\\q>Q}}\sum_{Q<p'\leq p-2}\|\nabla(\theta_2)_{p'}\|_\infty \|R^\perp w_p\|_{\ell}\|w_q\|_{\ell}^{\ell-1}\\
&\lesssim \ell\sum_{p'> Q}\|\nabla(\theta_2)_{p'}\|_\infty\sum_{p\geq p'+2}\|R^\perp w_p\|_{\ell}\sum_{\substack{|p-q|\leq2\\q>Q}} \|w_q\|_{\ell}^{\ell-1}\\
&\lesssim \ell\sum_{p'> Q}\lambda_{p'}^{1+\frac2r}\|(\theta_2)_{p'}\|_r\sum_{p\geq p'+2}\|R^\perp w_p\|_{\ell}\sum_{\substack{|p-q|\leq2\\q>Q}} \|w_q\|_{\ell}^{\ell-1},
\end{split}
\]
and then the definition of $\lb_{\theta,r}$ and Young's inequality to infer
\[
\begin{split}
I_{22} &\lesssim c_{\alpha,r}\nu \ell\sum_{p'> Q}\lambda_{p'}^\alpha\sum_{p\geq p'+2}\|R^\perp w_p\|_{\ell}\sum_{\substack{|p-q|\leq2\\q>Q}} \|w_q\|_{\ell}^{\ell-1}\\
&\lesssim c_{\alpha,r}\nu \ell\sum_{q> Q}\lambda_{q}^\alpha \|w_q\|_{\ell}^{\ell}.
\end{split}
\]
Since $r \geq \ell\geq 2$, we can choose $m$ so that $\frac1r+\frac1m+\frac{1}{\ell}=1$. To estimate $I_{3}$ we first integrate by parts, change the order of summations, and
use H\"older's inequality:
\[
\begin{split}
|I_{3}|&\leq \ell\sum_{q>Q}\sum_{\substack{p\geq q-2\\p>Q-1}}\int_{\mathbb T^2}|\Delta_q(R^\perp\tilde w_p (\theta_2)_p)\cdot \nabla (w_q|w_q|^{\ell-2})| \, dx\\
&\lesssim \ell^2\sum_{q>Q}\sum_{\substack{p\geq q-2\\p>Q-1}}\int_{\mathbb T^2}|\Delta_q(R^\perp\tilde w_p (\theta_2)_p) \cdot \nabla w_q||w_q|^{\ell-2} \, dx\\
&\leq \ell^2\sum_{p>Q}\sum_{Q<q\leq p+2}\int_{\mathbb T^2}\left|\Delta_q(R^\perp\tilde w_p (\theta_2)_p) \cdot \nabla w_q\right| |w_q|^{\ell-2} \, dx\\
&+\ell^2\sum_{Q<q\leq Q+2}\int_{\mathbb T^2}\left|\Delta_q(R^\perp\tilde w_Q (\theta_2)_Q) \cdot \nabla w_q\right| |w_q|^{\ell-2} \, dx\\
&\lesssim \ell^2\sum_{p> Q}\|R^\perp\tilde w_p\|_{\ell}\|(\theta_2)_p\|_r\sum_{Q<q\leq p+2}\lambda_q\|w_q\|_{(\ell-1)m}^{\ell-1}\\
&+\ell^2\|(\theta_2)_Q\|_\infty\|R^\perp\tilde w_Q\|_{\ell}\sum_{Q<q\leq Q+2}\lambda_q\|w_q\|_{\ell}^{\ell-1}.
\end{split}
\]
Then using the definition of $\lb_{\theta_2,r}$, the fact $w_{\leq Q}=0$ which implies $\tilde w_Q=w_{Q+1}$, Jensen's inequality, and \eqref{eq:Riesz}, we get
\[
\begin{split}
|I_{3}|&\lesssim c_{\alpha,r}\nu \ell^2\sum_{p> Q}\lambda_p^{\alpha-1-\frac2r}\|R^\perp w_p\|_{\ell}\sum_{Q<q\leq p+2}\lambda_q^{1+(\ell-1)[\frac2{\ell}-\frac{2}{(\ell-1)m}]}\|w_q\|_{\ell}^{\ell-1}\\
&+ c_{\alpha,r}\nu \ell^2\lambda_Q^{\alpha-1}\|R^\perp\tilde w_Q\|_{\ell}\sum_{Q<q\leq Q+2}\lambda_q\|w_q\|_{\ell}^{\ell-1}\\
&\lesssim c_{\alpha,r}\nu \ell^2\sum_{p> Q}\lambda_p^{\frac\alpha \ell}\|w_p\|_\ell\sum_{Q<q\leq p+2}\lambda_q^{\frac{\alpha(\ell-1)}{\ell}}\|w_q\|_\ell^{\ell-1}\lambda_{q-p}^{1-\alpha+\frac 2r+\frac\alpha \ell}\\
&+ c_{\alpha,r}\nu \ell^2\sum_{Q< q\leq Q+2}\lambda_q^\alpha\|w_q\|_\ell^\ell\\
&\lesssim c_{\alpha,r}\nu \ell^2\left(1-2^{\alpha-1-\frac \alpha \ell-\frac2r}\right)^{-\frac{\ell}{\ell-1}}\sum_{q> Q}\lambda_q^\alpha\|w_q\|_\ell^\ell,
\end{split}
\]
where we used  $\frac2r+\frac\alpha \ell>\alpha-1$ in order to apply Jensen's inequality. Therefore, we have
\begin{equation}\label{est-i1}
\begin{split}
|I| &\lesssim c_{\alpha,r}\nu \ell^2\left[\left(1-2^{\frac\alpha \ell-\frac2r}\right)^{-\ell}+\left(1-2^{\alpha-1-\frac \alpha \ell-\frac2r}\right)^{-\frac{\ell}{\ell-1}}\right]\sum_{q> Q}\lambda_q^\alpha\|w_q\|_\ell^\ell\\
&\lesssim c_{\alpha,r}\nu \ell^2\left(1-2^{\frac\alpha \ell-\frac2r}\right)^{-\ell}\sum_{q> Q}\lambda_q^\alpha\|w_q\|_\ell^\ell
\end{split}
\end{equation}
due to the choice of the parameters $\ell$ and $r$ as in \eqref{eq:defofl} and \eqref{para-r}.

To estimate $J$, we start with Bony's paraproduct formula
\begin{equation}\notag
\begin{split}
J=
&-\ell\sum_{q\geq -1}\sum_{|q-p|\leq 2}\int_{\T^2}\Delta_q((u_1)_{\leq{p-2}}\cdot\nabla w_p) w_q|w_q|^{\ell-2}\, dx\\
&-\ell\sum_{q\geq -1}\sum_{|q-p|\leq 2}\int_{\T^2}\Delta_q((u_1)_{p}\cdot\nabla w_{\leq{p-2}})w_q|w_q|^{\ell-2}\, dx\\
&-\ell\sum_{q\geq -1}\sum_{p\geq q-2}\int_{\T^2}\Delta_q((u_1)_p\cdot\nabla\tilde w_p)w_q|w_q|^{\ell-2}\, dx\\
=&J_{1}+J_{2}+J_{3}.
\end{split}
\end{equation}
Using the commutator notation \eqref{eq:CommutatodDef}, $J_{1}$ can be decomposed as 
\begin{equation}\notag
\begin{split}
J_{1}=&-\ell\sum_{q\geq -1}\sum_{|q-p|\leq 2}\int_{\T^2}[\Delta_q, (u_1)_{\leq{p-2}}\cdot\nabla] w_pw_q|w_q|^{\ell-2}\, dx\\
&-\ell\sum_{q\geq -1}\int_{\T^2} (u_1)_{\leq q-2}\cdot\nabla w_q w_q|w_q|^{\ell-2}\, dx\\
&-\ell\sum_{q\geq -1}\sum_{|q-p|\leq 2}\int_{\T^2}( (u_1)_{\leq{p-2}}- (u_1)_{\leq q-2})\cdot\nabla\Delta_qw_p w_q|w_q|^{\ell-2}\, dx\\
=&J_{11}+J_{12}+J_{13},
\end{split}
\end{equation}
where we used the fact that $\sum_{|p-q|\leq 2}\Delta_qw_p=w_q$, see (\ref{near-sum}).  Notice that we have $J_{12}=0$, since $\div\, (u_1)_{\leq q-2}=0$. 
Thanks to \eqref{commu},   
\begin{equation}\notag
\|[\Delta_q, (u_1)_{\leq{p-2}}\cdot\nabla] w_p\|_\ell\\
\lesssim \|\nabla  (u_1)_{\leq p-2}\|_\infty\|w_p\|_\ell.
\end{equation}
Thus, the term $J_{11}$ can be estimated as 
\begin{equation}\notag
\begin{split}
|J_{11}|
&\leq  \ell\sum_{q> Q} \sum_{\substack{|q-p|\leq 2\\ p>Q}}\int_{\mathbb T^2}\left| [\Delta_q, (u_1)_{\leq{p-2}}\cdot\nabla] w_p\right||w_q|^{\ell-1} \, dx\\
&\leq \ell\sum_{q> Q}\sum_{\substack{|q-p|\leq 2\\ p>Q}}\|[\Delta_q, (u_1)_{\leq{q}}\cdot\nabla] w_p\|_\ell\|w_q\|_\ell^{\ell-1}\\
&\leq \ell\sum_{q> Q}\sum_{\substack{|q-p|\leq 2\\ p>Q}}\|\nabla (u_1)_{(Q,q]}\|_\infty\|w_p\|_\ell\|w_q\|_\ell^{\ell-1}\\
&+\ell\sum_{q> Q}\sum_{\substack{|q-p|\leq 2\\ p>Q}}\|\nabla (u_1)_{\leq Q}\|_\infty\|w_p\|_\ell\|w_q\|_\ell^{\ell-1}\\
&= J_{111}+J_{112}.
\end{split}
\end{equation}
For the first term we use H\"older's and Bernstein's inequalities, 
\begin{equation}\notag
\begin{split}
J_{111}&\lesssim \ell\sum_{q> Q}\sum_{\substack{|q-p|\leq 2\\ p>Q}}\sum_{Q<p'\leq q}\lambda_{p'}\|(u_1)_{p'}\|_\infty\|w_p\|_\ell\|w_q\|_\ell^{\ell-1}\\
&\lesssim \ell\sum_{q> Q}\sum_{\substack{|q-p|\leq 2\\ p>Q}}\sum_{Q<p'\leq q}\lambda_{p'}^{1+\frac2r}\|(u_1)_{p'}\|_r\|w_p\|_\ell\|w_q\|_\ell^{\ell-1},
\end{split}
\end{equation}
and then the fact that $\| u_1\|_r\lesssim \|\theta_1\|_r$, and the definition of $\lb_{\theta_1,r}$ to get
\[
\begin{split}
J_{111}
&\lesssim c_{\alpha,r}\nu \ell\sum_{q> Q}\sum_{\substack{|q-p|\leq 2\\ p>Q}}\|w_p\|_\ell\|w_q\|_\ell^{\ell-1}\sum_{Q<p'\leq q}\lambda_{p'}^{\alpha}\\
&\lesssim c_{\alpha,r}\nu \ell\sum_{q> Q}\|w_q\|_\ell^{\ell}\sum_{Q<p'\leq q}\lambda_{p'}^{\alpha}\\
&\lesssim c_{\alpha,r}\nu \ell\sum_{q> Q}\lambda_q^{\alpha}\|w_q\|_\ell^{\ell}\sum_{Q<p'\leq q}\lambda_{p'-q}^{\alpha}\\
&\lesssim c_{\alpha,r}\nu \ell\sum_{q>Q}\lambda_q^{\alpha}\|w_q\|_\ell^{\ell}.
\end{split}
\]
The second term is estimated in a similar way:
\begin{equation}\notag
\begin{split}
J_{112} &\lesssim \ell\sum_{q> Q}\sum_{\substack{|q-p|\leq 2\\ p>Q}}\sum_{p'\leq Q}\lambda_{p'}\| (u_1)_{p'}\|_\infty\|w_p\|_\ell\|w_q\|_\ell^{\ell-1}\\
&\lesssim \ell\sum_{q> Q}\sum_{\substack{|q-p|\leq 2\\ p>Q}}\sum_{p'\leq Q}\lambda_{p'}\| (\theta_1)_{p'}\|_\infty\|w_p\|_\ell\|w_q\|_\ell^{\ell-1}\\
&\lesssim  c_{\alpha,r}\nu \ell\sum_{q> Q}\sum_{\substack{|q-p|\leq 2\\ p>Q}}\lambda_{Q}^\alpha\|w_p\|_\ell\|w_q\|_\ell^{\ell-1}\\
&\lesssim  c_{\alpha,r}\nu \ell\sum_{q>Q}\lambda_q^{\alpha}\|w_q\|_\ell^\ell.
\end{split}
\end{equation}
To estimate $J_{13}$, we start with splitting the summation
\begin{equation}\notag
\begin{split}
|J_{13}|&\leq \ell\sum_{q>Q}\sum_{\substack{|q-p|\leq 2\\p>Q}}\int_{\T^2}\left|( (u_1)_{\leq{p-2}}- (u_1)_{\leq q-2})\cdot\nabla\Delta_qw_p \right||w_q|^{\ell-1}\, dx\\
&\lesssim \ell\sum_{q>Q}\sum_{\substack{|q-p|\leq 2\\p>Q}}\sum_{q-3\leq p'\leq q}\int_{\T^2}|(u_1)_{p'}||\nabla\Delta_qw_p ||w_q|^{\ell-1}\, dx\\
&\lesssim \ell\sum_{q>Q}\sum_{\substack{|q-p|\leq 2\\p>Q}}\sum_{q-3\leq p'\leq Q}\int_{\T^2}|(u_1)_{p'}||\nabla\Delta_qw_p ||w_q|^{\ell-1}\, dx\\
&+ \ell\sum_{q>Q}\sum_{\substack{|q-p|\leq 2\\p>Q}}\sum_{\substack{q-3\leq p'\leq q\\p'>Q}}\int_{\T^2}|(u_1)_{p'}||\nabla\Delta_qw_p ||w_q|^{\ell-1}\, dx\\
&= J_{131}+J_{132}.
\end{split}
\end{equation}
We use H\"older's inequality for the first term
\[
J_{131} \lesssim \ell\sum_{q>Q}\|w_q\|_\ell^{\ell-1}\sum_{\substack{|q-p|\leq 2\\p>Q}}\lambda_p\|w_p\|_\ell\sum_{q-3\leq p'\leq Q}\|(u_1)_{p'}\|_\infty,
\]
followed  by the definition of $\lb_{\theta_1,r}$, H\"older's inequality, Young's inequality, and the fact $\alpha>0$
\begin{equation}\notag
\begin{split}
J_{131} &\lesssim  c_{\alpha,r}\nu \ell\sum_{q>Q}\lambda_Q^{\alpha}\lambda_q^{-1}\|w_q\|_\ell^{\ell-1}\sum_{\substack{|q-p|\leq 2\\p>Q}}\lambda_p\|w_p\|_\ell\\
&\lesssim  c_{\alpha,r}\nu \ell\sum_{q>Q}\lambda_q^{\frac{\alpha(\ell-1)}{\ell}}\|w_q\|_\ell^{\ell-1}\sum_{\substack{|q-p|\leq 2\\p>Q}}\lambda_p^{\frac{\alpha}{\ell}}\|w_p\|_\ell\lambda_{p-q}^{1-\frac{\alpha}{\ell}}\lambda_{Q-q}^\alpha\\
&\lesssim  c_{\alpha,r}\nu \ell\sum_{q>Q}\lambda_q^{\frac{\alpha(\ell-1)}{\ell}}\|w_q\|_\ell^{\ell-1}\sum_{\substack{|q-p|\leq 2\\p>Q}}\lambda_p^{\frac{\alpha}{\ell}}\|w_p\|_\ell\\
&\lesssim  c_{\alpha,r}\nu \ell\sum_{q>Q}\lambda_Q^{\alpha}\|w_q\|_\ell^{\ell}+ c_{\alpha,r}\nu \ell \sum_{q>Q}\sum_{\substack{|q-p|\leq 2\\p>Q}}\lambda_p^{\alpha}\|w_p\|_\ell^\ell\\
&\lesssim  c_{\alpha,r}\nu \ell\sum_{q>Q}\lambda_q^{\alpha}\|w_q\|_\ell^{\ell}.
\end{split}
\end{equation}
Since $r\geq \ell$ we can choose $m$ so that  $\frac1m+\frac1r+\frac{\ell-1}\ell=1$, and estimate the second term as
\begin{equation}\notag
\begin{split}
J_{132} &\lesssim \ell\sum_{q>Q}\|w_q\|_\ell^{\ell-1}\sum_{\substack{|q-p|\leq 2\\p>Q}}\lambda_p\|w_p\|_m\sum_{\substack{q-3\leq p'\leq q\\p'>Q}}\|(u_1)_{p'}\|_r\\
&\lesssim c_{\alpha,r}\nu \ell\sum_{q>Q}\|w_q\|_\ell^{\ell-1}\sum_{\substack{|q-p|\leq 2\\p>Q}}\lambda_p^{1+\frac2\ell-\frac2m}\|w_p\|_\ell\sum_{\substack{q-3\leq p'\leq q\\p'>Q}}\lambda_{p'}^{\alpha-1-\frac2r}\\
&\lesssim c_{\alpha,r}\nu \ell\sum_{q>Q}\lambda_q^{1+\frac2\ell-\frac2m}\|w_q\|_\ell^{\ell}\sum_{\substack{q-3\leq p'\leq q\\p'>Q}}\lambda_{p'}^{\alpha-1-\frac2r}\\
&\lesssim c_{\alpha,r}\nu \ell\sum_{q>Q}\lambda_q^{\alpha}\|w_q\|_\ell^{\ell}\sum_{\substack{q-3\leq p'\leq q\\p'>Q}}\lambda_{p'-q}^{\alpha-1-\frac2r}\\
&\lesssim c_{\alpha,r}\nu \ell\sum_{q> Q}\lambda_q^{\alpha}\|w_q\|_\ell^\ell.
\end{split}
\end{equation}
Again choosing $m$ such that  $\frac1r+\frac1m+\frac{\ell-1}{\ell}=1$ and using H\"older's inequality, we obtain
\begin{equation}\notag
\begin{split}
|J_{2}| &\leq \ell\sum_{q>Q}\sum_{\substack{|q-p|\leq 2\\p>Q+2}}\int_{\T^2}\left|\Delta_q((u_1)_{p}\cdot\nabla w_{\leq{p-2}})\right||w_q|^{\ell-1}\, dx\\
&\leq \ell\sum_{p>Q+2}\sum_{\substack{|q-p|\leq 2\\q>Q}}\int_{\T^2}\left|\Delta_q((u_1)_{p}\cdot\nabla w_{\leq{q}})\right||w_q|^{\ell-1}\, dx\\
&\lesssim \ell\sum_{p>Q+2}\|(u_1)_p\|_r\sum_{\substack{|q-p|\leq 2\\q>Q}}\|w_q\|_\ell^{\ell-1}\sum_{p'\leq q}\lambda_{p'}\|w_{p'}\|_m.
\end{split}
\end{equation}
Now we use the definition of $\lb_{\theta_1,r}$ and Jensen's inequality to conclude that
\begin{equation}\notag
\begin{split}
|J_{2}|
&\lesssim c_{\alpha,r}\nu \ell\sum_{p>Q+2}\lambda_p^{\alpha-1-\frac 2r}\sum_{\substack{|q-p|\leq 2\\q>Q}}\|w_q\|_\ell^{\ell-1}\sum_{p'\leq q}\lambda_{p'}\|w_{p'}\|_m\\
&\lesssim c_{\alpha,r}\nu \ell \sum_{q>Q}\lambda_q^{\alpha-1-\frac 2r}\|w_q\|_\ell^{\ell-1}\sum_{p'\leq q}\lambda_{p'}^{1+\frac 2\ell-\frac2m}\|w_{p'}\|_\ell\\
&\lesssim c_{\alpha,r}\nu \ell \sum_{q>Q}\lambda_q^{\frac{\alpha(\ell-1)}{\ell}}\|w_q\|_\ell^{\ell-1}\sum_{p'\leq q}\lambda_{p'}^{\frac\alpha \ell}\|w_{p'}\|_\ell\lambda_{p'-q}^{1+\frac 2r-\frac\alpha \ell}\\
&\lesssim c_{\alpha,r}\nu \ell \left(1-2^{\frac\alpha \ell-1-\frac2r}\right)^{-l}\sum_{q>Q}\lambda_q^\alpha\|w_q\|_\ell^\ell,
\end{split}
\end{equation}
where we used $1+\frac 2r-\frac\alpha \ell>0$. Finally, observe that $J_{3}$ enjoys the same estimate as $I_{3}$ due to the fact that $\|(u_1)_q\|_{r_1}\lesssim \|(\theta_1)_q\|_{r_1}$ for any $r_1\in(1,\infty]$. Thus
\begin{equation}\label{est-i2}
\begin{split}
|J| &\lesssim c_{\alpha,r}\nu \ell^2\left[\left(1-2^{\frac\alpha \ell-1-\frac2r}\right)^{-l}+\left(1-2^{\alpha-1-\frac \alpha \ell-\frac2r}\right)^{-l}\right]\sum_{q> Q}\lambda_q^\alpha\|w_q\|_\ell^\ell\\
&\lesssim c_{\alpha,r}\nu \ell^2\left(1-2^{\frac\alpha \ell-\frac2r}\right)^{-l}\sum_{q> Q}\lambda_q^\alpha\|w_q\|_\ell^\ell
\end{split}
\end{equation}
due to \eqref{eq:defofl} and \eqref{para-r}. It is worth to mention that the coefficient $\left(1-2^{\frac\alpha \ell-\frac2r}\right)^{-l}$ blows up when $r$ approaches $\frac 4{\alpha-1}$ in the case $\alpha>1$ and $\ell=\frac{2\alpha}{\alpha-1}$.

Combining (\ref{w2})--(\ref{est-i2}) yields
\begin{equation}\notag
\|w(t)\|_{B^0_{\ell,\ell}}^\ell - \|w(t_0)\|_{B^0_{\ell,\ell}}^\ell \leq \int_{t_0}^t\left( - C \nu\|\Lambda^{\alpha/\ell} w\|_{B^0_{\ell,\ell}}^\ell
+C_1c_{\alpha,r}\nu \ell^2 \left(1-2^{\frac\alpha \ell-\frac2r}\right)^{-l}\sum_{q>Q}\lambda_q^\alpha\|w_q\|_\ell^\ell \right)\, d\tau,
\end{equation}
for some absolute constants $C$ and $C_1$. Recall that $\ell$ is defined as in (\ref{eq:defofl}), and
\[
c_{\alpha,r}=\left\{
\begin{split}
\frac{c_0}{\alpha^2(r+1)^2}\left(1-2^{\frac2{r+1}-\frac2r}\right)^{\frac{\alpha(r+1)}{2}},&\qquad 0<\alpha\leq 1,\\
c_0(\alpha-1)^2\left(1-2^{\frac{\alpha-1}2-\frac2r}\right)^{\frac{2\alpha}{\alpha-1}}, &\qquad  1<\alpha< 2.
\end{split}
\right.
\]
Hence, choosing $c_0=\frac{C}{2C_1}$, we arrive at
\[
\|w(t)\|_{B^0_{\ell,\ell}}^\ell - \|w(t_0)\|_{B^0_{\ell,\ell}}^\ell \leq - \frac{C\nu}{2} \int_{t_0}^t \|\Lambda^{\alpha/\ell} w\|_{B^0_{\ell,\ell}}^\ell\, d\tau
\lesssim  - \lambda_0^\alpha \nu \int_{t_0}^t \| w\|_{B^0_{\ell,\ell}}^\ell\, d\tau,
\]
for all $T_1 \leq t_0 \leq t \leq T_2$. Combining it with Gr\"onwall's inequality gives the desired result.
\end{proof}

Clearly, Theorem \ref{thm-main-bound} implies the first part of Theorem \ref{thm}. To prove the second part, we need to introduce viscosity solutions and show that
the global attractor for such solutions  is bounded in $L^\infty$.

\bigskip

\section{$L^\infty$ estimates}
\label{sec:infty}

The goal of this section is to obtain an explicit $L^\infty$ bound on viscosity solutions to \eqref{QG} when the force $f$ is in $L^p$ for some $p>2/\alpha$.

A weak solution $\th(t)$ on $[0,T]$ is called a viscosity solution if there exist sequences $\e_n \to 0$ and $\th_n(t)$ satisfying
\begin{equation}\begin{split}\label{VQG}
\frac{\partial\theta_n}{\partial t}+u_n\cdot\nabla \theta_n+\nu\Lambda^\alpha\theta_n + \e_n \Delta \th_n=f,\\
u_n=R^\perp\theta_n,
\end{split}
\end{equation}
such that $\th_n \to \theta$ in $C_\mathrm{w}([0,T];L^2)$. Standard arguments imply that for any initial data $\th_0 \in L^2$ there exists a viscosity solution $\th(t)$ of \eqref{QG} on $[0,\infty)$ with $\th(0)=\th_0$ (see \cite{CC}, for example). The solution $\theta(t)$ may
enjoy some regularity depending on the force, but this is not needed for our argument.

In the case of $\alpha=1$ and zero force, Caffarelli and Vasseur derived a level set energy inequality using a harmonic extension \cite{CaV}. Here we sketch a modification of the proof from \cite{CD} extended to all  $\alpha>0$.

\begin{Lemma}
\label{le:existence}
Let $\alpha>0$ and $\th(t)$ be a viscosity solution to \eqref{QG} on $[0,T]$ with $\th(0) \in L^2$. Then
for every $\gamma \in \mathbb{R}$ it satisfies the level set energy inequality
\begin{equation}\label{truncated}
\frac{1}{2}\|\tilde\theta_\gamma(t_2)\|_2^2+\nu\int_{t_1}^{t_2}\|\Lambda ^{\frac{\alpha}{2}}\tilde\theta_\gamma\|_2^2 \, dt
\leq \frac{1}{2}\|\tilde\theta_\gamma(t_1)\|_2^2+\int_{t_1}^{t_2}\int_{\mathbb T^2}f\tilde\theta_\gamma \, dxdt,
\end{equation}
for all $t_2\in[t_1, T]$ and a.e. $t_1\in [0,T]$. Here $\tilde\theta_\gamma=(\theta-\gamma)_+$ or $\tilde\theta_\gamma=(\theta+\gamma)_{-}$.
\end{Lemma}

\begin{proof}

We only show a priori estimates. It is clear how to pass to the limit in \eqref{VQG} as $\e \to 0$.
Denote $\varphi(\theta)=(\theta-\gamma)_+$. Note that $\varphi$ is Lipschitz and 
\begin{equation}\notag
\varphi'(\theta)\varphi(\theta)=\varphi(\theta).
\end{equation}
Multiplying the first equation of (\ref{QG}) by $\varphi'(\theta)\varphi(\theta)$ and integrating over $\mathbb T^2$ yields
\begin{equation}\label{energy1}
\begin{split}
\frac{1}{2}\frac{d}{dt}\int_{\mathbb T^2}\varphi^2(\theta)\,dx+\int_{\mathbb T^2}\nabla\cdot\left(\frac{1}{2}\varphi^2(\theta)u\right)\,dx\\
+\nu\int_{\mathbb T^2}\Lambda^\alpha\theta\varphi(\theta)\,dx
=\int_{\mathbb T^2}f\varphi(\theta)\,dx.
\end{split}
\end{equation}
Let $f, g \in C^{\infty}(\mathbb{T}^2)$ such that $g(x) = (f(x)-\gamma)_+$. Then one can easily verify that
\[
(f(x)-f(y))(g(x)-g(y)) \geq (g(x)-g(y))^2.
\]
Now, by Fubini's Theorem, 
\[
\begin{split} 
&\sum_{j\in \mathbb Z^2}\int_{\mathbb T^2} P.V. \int_{\mathbb T^2}\frac{f(x)-f(y)}{|x-y+Lj|^{2+\alpha}}g(x)\, dy \, dx=\\
&\sum_{j\in \mathbb Z^2}\int_{\mathbb T^2} P.V. \int_{\mathbb T^2}\frac{f(y)-f(x)}{|x-y+Lj|^{2+\alpha}}g(y)\, dy \, dx.
\end{split}
\]
Note that (see \cite{CC})
\[
\Lambda^\alpha f=  \frac{c_{\alpha,r}}{2}\sum_{j\in \mathbb Z^2} P.V. \int_{\mathbb T^2}\frac{f(x)-f(y)}{|x-y+Lj|^{2+\alpha}}\, dy. 
\]
Therefore
\[
\begin{split}
\int_{\mathbb T^2}\Lambda^\alpha f g\, dx &=
\frac{c_{\alpha,r}}{2}\sum_{j\in \mathbb Z^2}\int_{\mathbb T^2} P.V. \int_{\mathbb T^2}\frac{(f(x)-f(y))(g(x)-g(y))}{|x-y+Lj|^{2+\alpha}}\, dy \, dx\\
&\geq \frac{c_{\alpha,r}}{2}\sum_{j\in \mathbb Z^2}\int_{\mathbb T^2} \int_{\mathbb T^2}\frac{(g(x)-g(y))^2}{|x-y+Lj|^{2+\alpha}}\, dy \, dx\\
&=\int_{\mathbb T^2}\left|\Lambda^{\frac{\alpha}{2}} g(x)\right|^2\,dx.
\end{split}
\]
Clearly this inequality also holds for $\theta$ and $\varphi(\theta)$, giving
\[
\int_{\mathbb T^2}\Lambda^\alpha\theta\varphi(\theta)\,dx \geq \int_{\mathbb T^2}\left|\Lambda^{\frac{\alpha}{2}}\varphi(\theta)\right|^2dx
\]
Thus, it follows from (\ref{energy1}) that
\begin{equation}\notag
\begin{split}
\frac{1}{2}\frac{d}{dt}\int_{\mathbb T^2}\varphi^2(\theta)dx+\int_{\mathbb T^2}\nabla\cdot\left(\frac{1}{2}\varphi^2(\theta)u\right)\,dx\\
+\nu\int_{\mathbb T^2}\left|\Lambda^{\frac{\alpha}{2}}\varphi(\theta)\right|^2\,dx
\leq\int_{\mathbb T^2}f\varphi(\theta)\,dx.
\end{split}
\end{equation}
Since the integral $\int_{\mathbb T^2}\nabla\cdot\left(\frac{1}{2}\varphi^2(\theta)u\right)dx=0$,
this gives us the truncated energy inequality (\ref{truncated}). The case of $\varphi(\theta)=(\theta+\gamma)_-$ is similar.

\end{proof}

Now we can use De Giorgi iteration to obtain explicit bounds on the $L^\infty$ norm. For $\alpha=1$ this was done in \cite{CaV} in the unforced case $f=0$, and similarly
in \cite{CD} for $f\in L^p$. $p>2$. Here we extend the proof in \cite{CD} to cover the whole range $\alpha>0$.

\begin{Lemma}
\label{Linfty}
Let $\alpha \in (0, 2)$ and $\theta$ be a viscosity solution of \eqref{QG} on $[0,\infty)$ with $\th(0) \in L^2$ and $f\in L^p(\mathbb T^2)$ for some $p\in(2/\alpha,\infty]$. Then,
for every $t>0$,
\begin{equation}\label{infty-norm}
\|\theta(t)\|_{L^\infty}\lesssim  \frac{\|\theta(0)\|_{2}}{(\nu t)^{\frac1\alpha}}+
\lambda_0^{\frac2p -\alpha} \frac{\|f\|_p}{\nu}\left(1+\lambda_0^{\frac{\alpha}{2}-1}(\nu t)^{\frac{1}{2}-\frac{1}{\alpha}}\right).
\end{equation}
\end{Lemma}
\begin{proof} Consider the levels 
\begin{equation}\notag
\gamma_k=M(1-2^{-k})
\end{equation}
for some $M$ to be determined later,
and denote the truncated function 
\[
\theta_k=(\theta-\gamma_k)_+.
\]
Fix $t_0>0$. Let $T_k=t_0(1-2^{-k})$ and  define the energy levels as:
\begin{equation}\notag
U_k=\sup_{T_k\leq t \leq t_0}\|\theta_k(t)\|_2^2+2\nu\int_{T_k}^{t_0}\|\Lambda^{\frac{\alpha}{2}}\theta_k(t)\|_2^2 \,dt.
\end{equation}
We take $\tilde\theta=\theta_k$ and $t_1=s \in(T_{k-1},T_k)$, $t_2=t>T_k$ in the truncated energy inequality (\ref{truncated}).
Then taking $t_1=s \in(T_{k-1},T_k)$, $t_2=T>t$, adding the two inequalities, taking $\limsup_{T\to t_0}$ and then $\sup_{t\in[T_k, t_0]}$ gives 
\begin{equation}\notag
U_k\leq 2\|\theta_k(s)\|_2^2+2\int_{T_{k-1}}^{t_0} \int_{\mathbb T^2}\left|f(x) \theta_k(x,\tau)\right|\, dxd\tau,
\end{equation}
for a.a. $s \in (T_{k-1}, T_k)$.
Taking the average in $s$ on $[T_{k-1}, T_k]$ yields
\begin{equation}\label{truncated1}
U_k\leq \frac{2^{k+1}}{t_0}\int_{T_{k-1}}^{t_0}\int_{\mathbb T^2} \theta_k^2(x,s)\, dxds+2\int_{T_{k-1}}^{t_0} \int_{\mathbb T^2}\left|f(x) \theta_k(x,t)\right|\, dxdt.
\end{equation}

Interpolating between $L^\infty(L^2)$ and $L^2(H^{\frac{\alpha}{2}})$, we get
\begin{equation}\label{interpolation}
\int_{T_{k}}^{t_0}\int_{\mathbb T^2} |\theta_{k}|^{2+\alpha}\,dxdt \leq
 \frac{C}{\nu} U_k^{\frac{2+\alpha}{2}},
\end{equation}
for some absolute constant $C$.

Note that
\begin{equation}\notag
\theta_{k-1}\geq 2^{-k}M \qquad \mbox { on } \left\{(x,t):\theta_k(x,t)>0\right\},
\end{equation}
and hence
\begin{equation}\notag
1_{\left\{\theta_k>0\right\}}\leq \frac{2^{k\alpha}}{M^\alpha}\theta_{k-1}^\alpha.
\end{equation}
Therefore, using the fact that $\theta_k\leq\theta_{k-1}$ and (\ref{interpolation}), we get
\begin{equation}\label{truncated2}
\begin{split}
\frac{2^{k+1}}{t_0}\int_{T_{k-1}}^{t_0}\int_{\mathbb T^2} \theta_k^2(x,s)\,dxds
\leq &\frac{2^{k+1}}{t_0}\int_{T_{k-1}}^{t_0}\int_{\mathbb T^2} \theta_{k-1}^2(x,s)1_{\left\{\theta_k>0\right\}}\, dxds\\
\leq &\frac{2^{k(\alpha+1)+1}}{t_0M^\alpha}\int_{T_{k-1}}^{t_0}\int_{\mathbb T^2} |\theta_{k-1}|^{2+\alpha}\, dxds\\
\leq &C\frac{2^{k(\alpha+1)+1}}{\nu t_0M^\alpha}U_{k-1}^{\frac{2+\alpha}2}.
\end{split}
\end{equation}
On the other hand, since $f\in L^p(\mathbb T^2)$ with $p>\frac2\alpha$, we obtain, for $p'=\frac{p}{p-1}$ (or $p'=1$ when $p=\infty$),
\begin{equation}\label{truncated3}
\begin{split}
&\int_{T_{k-1}}^{t_0} \int_{\mathbb T^2}\left|f(x) \theta_k(t)\right|dxdt\\
\leq &\|f\|_p\int_{T_{k-1}}^{t_0} \left(\int_{\mathbb T^2}|\theta_k|^{p'}dx\right)^{1/{p'}}dt\\
\leq &\|f\|_p\int_{T_{k-1}}^{t_0} \left(\int_{\mathbb T^2}|\theta_{k-1}|^{p'}1^{2+p'\alpha-p'}_{\left\{\theta_k>0\right\}}dx\right)^{1/{p'}}dt\\
\leq &\|f\|_p\frac{2^{k({2/{p'}+\alpha-1})}}{M^{2/{p'}+\alpha-1}}\int_{T_{k-1}}^{t_0} \left(\int_{\mathbb T^2}|\theta_{k-1}|^{2+p'\alpha}dx\right)^{1/{p'}}dt\\
\leq &\|f\|_p\frac{2^{k({2/{p'}}+\alpha-1)}}{M^{2/{p'}+\alpha-1}}\sup_{t\geq T_{k-1}}\left(\int_{\mathbb{T}^2}|\theta_{k-1}|^2\, dx\right)^{\frac{2-2p'+p'\alpha}{2p'}}\int_{T_{k-1}}^{t_0} \left(\int_{\mathbb T^2}|\theta_{k-1}|^{\frac4{2-\alpha}}dx\right)^{\frac{2-\alpha}2}dt\\
\leq &\|f\|_p\frac{2^{k({2/{p'}}+\alpha-1)}}{\nu M^{2/{p'}+\alpha-1}}U_{k-1}^{1+(2-2p'+p'\alpha)/{2p'}}.
\end{split}
\end{equation}
Note that $(2-2p'+p'\alpha)/{2p'}>0$, since $p>2/\alpha$.
Combining (\ref{truncated1}), (\ref{truncated2}) and (\ref{truncated3}) yields
\begin{equation}\label{iteration}
U_k\leq C\frac{2^{k(\alpha+1)+1}}{\nu t_0M^\alpha}U_{k-1}^{(2+\alpha)/2}+C\|f\|_p\frac{2^{k(2/{p'}+\alpha-1)}}{\nu M^{2/{p'}+\alpha-1}}U_{k-1}^{1+(2-2p'+p'\alpha)/(2p')}
\end{equation}
with a constant $C$ independent of $k$. We claim that for a large enough $M$, the above nonlinear iteration inequality implies that $U_k$ converges to 0 as $k\to \infty$.
Thus $\theta(t_0)\leq M$ for almost all $x$. The same argument applied to $\theta_k=(\theta+\gamma_k)_-$ also gives
a lower bound.

To prove the above claim (and automatically get an explicit expression for $M$ in terms of $t_0$ and $U_0$), first note that $\theta \leq 0$ almost everywhere if $U_0=0$. Assume now $U_0>0$. Denote $\delta=(2-2p'+p'\alpha)/(2p')$. Note that $0< \delta <\alpha/2$.  Define $V_k=\eta_k U_k$ with $\eta_k=2^{mk}\eta_0$. We choose 
\begin{equation}\label{para}
\begin{split}
m&=\max\left\{2+\frac2\alpha, \ \frac{2+p'(\alpha-1)}{p'\delta}\right\}, \qquad \eta_0=\frac1{2U_0},\\
 M&=\frac{(4C)^{\frac1\alpha}2^{m(\frac1\alpha+\frac12)}(2U_0)^{\frac12}}{(\nu t_0)^{1/\alpha}}+\left(\frac{2C2^{m(1+\delta)}\|f\|_p}{\nu}\right)^{\frac{p'}{2+p'(\alpha-1)}}(2U_0)^{\frac{p'\delta}{2+p'(\alpha-1)}}\\
&\sim U_0^{\frac12}(\nu t_0)^{-\frac1\alpha}+\left(\frac{\|f\|_p}{\nu}\right)^{\frac{p'}{2+p'(\alpha-1)}}U_0^{\frac{p'\delta}{2+p'(\alpha-1)}}.
\end{split}
\end{equation}
Based on the choice of the parameters $m, M, \eta_0$, one can verify that 
\[
C\eta_k\frac{2^{k(\alpha+1)+1}}{\nu t_0M^\alpha}U_{k-1}^{1+\alpha/2}\leq\frac12\eta_{k-1}^{1+\alpha/2}U_{k-1}^{1+\alpha/2}, \qquad k \geq 1;
\]
\[
C\eta_k\|f\|_p\frac{2^{k({2/{p'}}+\alpha-1)}}{\nu M^{2/{p'}+\alpha-1}}U_{k-1}^{1+\delta}\leq\frac12\eta_{k-1}^{1+\delta}U_{k-1}^{1+\delta}, \qquad k \geq 1.
\]
It follows from (\ref{iteration}) that
\begin{equation}\notag
\begin{split}
V_k=\eta_k U_k\leq &C\eta_k\frac{2^{k(\alpha+1)+1}}{\nu t_0M^\alpha}U_{k-1}^{1+\alpha/2}+C\eta_k\|f\|_p\frac{2^{k({2/{p'}}+\alpha-1)}}{\nu M^{2/{p'}+\alpha-1}}U_{k-1}^{1+\delta}\\
\leq&\frac12\eta_{k-1}^{1+\alpha/2}U_{k-1}^{1+\alpha/2}+\frac12\eta_{k-1}^{1+\delta}U_{k-1}^{1+\delta}\\
=&\frac12 V_{k-1}^{1+\alpha/2}+\frac12 V_{k-1}^{1+\delta},
\end{split}
\end{equation}
for all $k\geq 1$.
We also have $V_0=\eta_0 U_0<1/2$. Recalling that $0<\delta<\alpha/2$, we arrive at
\[
V_k\leq V_{k-1}^{1+\delta}, \qquad k\geq 1.
\]
It implies that  $V_k\to 0$ and hence $U_k\to 0$ as $k\to \infty$. Thus (\ref{para}) yields
\begin{equation} \label{eq:M}
\|\theta(t_0)\|_{\infty} \lesssim U_0^{\frac12}(\nu t_0)^{-\frac1\alpha}+\left(\frac{\|f\|_p}{\nu}\right)^{\frac{p}{p+p\alpha-2}}U_0^{\frac{p\alpha - 2}{2(p+p\alpha-2)}}.
\end{equation}

Now it follows from the energy inequality that
\begin{equation}\notag
\begin{split}
\|\theta(t_0)\|_2^2+2\nu\int_0^{t_0}\|\Lambda^{\frac{\alpha}{2}}\theta(t)\|_2^2\, dt\leq& \|\theta(0)\|_2^2+\int_0^{t_0}\int_{\T^2}f(x)\theta(x,t) \, dxdt\\
\leq &\|\theta(0)\|_2^2+\frac C\nu\int_0^{t_0}\|f\|_{H^{-\frac\alpha2}}^2\, dt+\nu\int_0^{t_0}\|\theta\|_{H^{\frac\alpha2}}^2\, dt.
\end{split}
\end{equation}
Applying Poincar\'e's and H\"older's inequalities, this leads to 
\begin{equation}\notag
\begin{split}
\|\theta(t_0)\|_2^2+\nu\int_0^{t_0}\|\Lambda^{\frac{\alpha}{2}}\theta(t)\|_2^2\, dt
\lesssim &\|\theta(0)\|_2^2+\frac 1\nu\int_0^{t_0}\|f\|_{H^{-\frac\alpha2}}^2\, dt\\
\lesssim &\|\theta(0)\|_2^2+\frac {t_0\lambda_0^{-\alpha}}\nu \|f\|_{2}^2\\
\lesssim &\|\theta(0)\|_2^2+\frac {t_0\lambda_0^{\frac 4p-2-\alpha}}\nu \|f\|_{p}^2,\\
\end{split}
\end{equation}
where $\lambda_0=L^{-1}$. Therefore, by the definition of $U_0$, we have
\begin{equation}\label{U01}
\begin{split}
U_0=&\sup_{0\leq t \leq t_0}\|\theta_+(t)\|_2^2+2\nu\int_{0}^{t_0}\|\Lambda^{\frac{\alpha}{2}}\theta_+(t)\|_2^2 \,dt\\
\lesssim &\|\theta(0)\|_2^2 + \frac{t_0\lambda_0^{\frac{4}{p}-2-\alpha}}{\nu}\|f\|_p^2.
\end{split}
\end{equation}
Applying Young's inequality to \eqref{eq:M} and combining it with (\ref{U01}) yields
\begin{equation} \label{eq:mainbound}
\begin{split}
\|\theta(t)\|_{\infty} \lesssim &U_0^{\frac12}(\nu t)^{-\frac1\alpha}+\left(\frac{\|f\|_p}{\nu}(\nu t)^{1-\frac2{p\alpha}}\right)^{\frac{p}{p+p\alpha-2}}\left(U_0^{\frac12}(\nu t)^{-\frac1\alpha}\right)^{\frac{p\alpha - 2}{p+p\alpha-2}}\\
\lesssim &U_0^{\frac12}(\nu t)^{-\frac1\alpha}+\frac{\|f\|_p}{\nu}(\nu t)^{1-\frac2{p\alpha}}+U_0^{\frac12}(\nu t)^{-\frac1\alpha}\\
\lesssim &\|\theta(0)\|_2 (\nu t)^{-\frac1\alpha}+\lambda_0^{\frac2p-1-\frac\alpha2}\|f\|_p\nu^{-\frac12-\frac1\alpha}t^{\frac12-\frac1\alpha}+\|f\|_p\nu^{-\frac2{p\alpha}}t^{1-\frac2{p\alpha}}\\
\lesssim &\|\theta(0)\|_2 (\nu t)^{-\frac1\alpha}+\lambda_0^{\frac2p-1-\frac\alpha2}\|f\|_p\nu^{-\frac12-\frac1\alpha}t^{\frac12-\frac1\alpha},
\end{split}
\end{equation}
for $0<t\leq \lambda_0^{-\alpha}\nu^{-1}$. We will now show that this bound holds for
all $t>0$. Indeed, take $t=T=\lambda_0^{-\alpha}\nu^{-1}$ in \eqref{eq:mainbound} and
then shift it by $t-T$, which gives
\begin{equation} \label{eq:mainbound2}
\|\theta(t)\|_\infty \lesssim \frac{\|\theta(t-T)\|_2}{(\nu T)^{\frac{1}{\alpha}}}+
\lambda_0^{\frac2p -\alpha} \frac{\|f\|_p}{\nu}.
\end{equation}
Thanks to the energy inequality, we also have
\[
\|\theta(t-T)\|_2^2 \lesssim \|\theta(0)\|_2^2e^{-\nu(2\pi\lambda_0)^\alpha (t-T)}+
\frac{\lambda_0^{-2\alpha-2+\frac4p}}{\nu^2} \|f\|_p^2.
\]
Combining this with \eqref{eq:mainbound2} we obtain
\[
\begin{split}
\|\theta(t)\|_\infty &\lesssim \frac{\|\theta(0)\|_2}{(\nu T)^{\frac{1}{\alpha}}}
e^{-\frac{\nu}2(2\pi\lambda_0)^\alpha (t-T)}+
\lambda_0^{\frac2p -\alpha} \frac{\|f\|_p}{\nu}\\
&\lesssim \frac{\|\theta(0)\|_2}{(\nu t)^{\frac{1}{\alpha}}}+
\lambda_0^{\frac2p -\alpha} \frac{\|f\|_p}{\nu},
\end{split}
\]
for $t\geq T$. Here we used the fact that $t^\beta e^{-\delta(t-1)}\leq \left(\frac{\beta}{\delta}\right)^\beta e^{\delta-\beta}$ on $(0,\infty)$.
\end{proof}

\section{Global attractor and  bounds on the determining wavenumber}
\subsection{Global attractor}
Thanks to the energy inequality, we have 
\[
\|\theta(t)\|_2^2 \leq \|\theta(0)\|_2^2e^{-\nu(2\pi\lambda_0)^\alpha t}+\frac{\|\Lambda^{-\frac{\alpha}{2}}f\|_2^2}{\nu^2(2\pi\lambda_0)^\alpha}\left(1-e^{-\nu(2\pi\lambda_0)^\alpha t }\right), \qquad t>0,
\]
where $\lambda_q= 2^q/L$ as before.
Denote
\[
B_{L^2} = \left\{ \theta \in L^2: \|\theta\|_2 \leq R_{2} \right\}, \qquad R_{2} = \frac{\|\Lambda^{-\frac{\alpha}{2}}f\|_2}{\nu\lambda_0^{\alpha/2}}.
\]
Then for any solution $\theta(t)$ there exists a time $t_{L^2}$ that depends only on $\|\theta(0)\|_2$, such that
\[
\theta(t) \in B_{L^2}, \qquad \forall t \geq t_{L^2}.
\]
So the set $B_{L^2}$ is an absorbing ball in $L^2$.
Moreover, there is a global attractor $\mathcal{A} \subset B_{L^2}$,
\[
\mathcal{A} = \{\theta(0): \theta(t) \text{ is a complete bounded trajectory, i.e., } \theta \in L^\infty((-\infty,\infty);L^2)\}.
\]
In \cite{CD}, in the critical case $\alpha=1$, we proved that $\mathcal{A}$ is a compact global attractor in the classical sense. It uniformly attracts bounded sets in $L^2$,
it is the minimal closed attracting set, and it is the $L^2$-omega limit of the absorbing ball $B_{L^2}$. This was done using the De Georgi iteration method to obtain
$L^2$ continuity of solutions (which is automatically true in the subcritical case $\alpha>1$), and applying the framework of evolutionary systems in \cite{C}. With all the
ingredients at hand, the framework \cite{C} gives the existence of the global attractor in the subcritical case as well. However, in the critical case $\alpha<1$, we only
know the existence of a weak global attractor at this point.

We also proved in \cite{CD} that the global attractor $\mathcal{A}$ is bounded in $L^\infty$.
Indeed, Lemma~\ref{Linfty} implies that
\[
\mathcal{A} \subset B_{L^\infty},
\]
where
\[
B_{L^\infty}=\left\{\theta \in B_{L^2}: \|\theta\|_\infty \leq  R_\infty \right\}, 
\]
and
\begin{equation} \label{eq:boundOnR}
R_\infty \sim \lambda_0^{\frac2p-\alpha}\frac{\|f\|_p}{\nu}.
\end{equation}
Moreover, for any solution $\theta(t)$ there exists time $t_{L^\infty}$ that depends only on $\|\theta(0)\|_2$,  such that
\[
\theta(t) \in B_{L^\infty}, \qquad \forall t\geq t_{L^\infty}.
\]
So $B_{L^\infty}$ is an absorbing set in $L^2$. 

\subsection{Proof of the second part of Theorem \ref{thm}}
\label{sec-pf-thm}
For a viscosity solution $\theta(t)$ on the global attractor, i.e., such that $\theta \in L^\infty((-\infty,\infty);L^2)$, we have
\[
\begin{split}
\|\theta(t)\|_{B^0_{\ell,\ell}} &\lesssim \|\theta(t)\|_\infty^{1-\frac{2}{\ell}} \|\theta(t)\|_2^{\frac{2}{\ell}}\\
&\lesssim R_{\infty}^{1-\frac{2}{\ell}} R_{2}^{\frac{2}{\ell}},
\end{split}
\]
for all $t$. Hence, fixing $t$ in \eqref{eq:boundonldif-expon} and taking a limit as $t_0$ goes to $-\infty$, we obtain the desired result.
\qed

 In the following two subsections we will derive explicit bounds on $\lb_\theta$ for solutions $\theta$ in the absorbing set $B_{L^\infty}$.

\subsection{The subcritical case $\alpha\in (1,2)$: proof of Theorem \ref{thm-theta-est1}}

In this case  $\lb_{\theta,r}$ is a determining wavenumber for all $r\in(\frac{2\alpha}{\alpha-1},\frac{4}{\alpha-1} )$. When  $r = \frac{4}{\alpha -1 }$ our estimates blow up. Nevertheless, we are able to pass to a limit as $r  \to \frac{4}{\alpha -1 }=:r_0$ and show that $\lb_{\theta,r} \leq \lb_{\theta}$ for  some $r<r_0$, where 
\[ 
\lb_{\theta}(t)=\min\{\lambda_q:\lambda_p^{\frac{1-\alpha}{2}}\|\theta_{p}\|_{\frac{4}{\alpha-1}}<{\textstyle \frac{c_{\alpha,r}}{2}}\nu \quad \forall p>q, \ \mbox {and}\  \lambda_q^{-\alpha}\sum_{p\leq q}\lambda_p\|\theta_{p}\|_{\infty}<c_{\alpha,r}\nu  \},
\]
thanks
to the following observation.

\begin{Lemma} \label{l:limit-in-r} Let $\alpha>1.$
There exists a function $r(M)\in(\frac{2\alpha}{\alpha-1},\frac{4}{\alpha-1} )$, such that 
\[
\lb_{\theta,r(M)}(t) \leq \lb_{\theta}(t), \qquad t\in [T_1,T_2],
\]
provided
\[
\|\theta\|_{L^\infty(\mathbb{T} \times [T_1,T_2])} < M.
\]
\begin{proof}
It is sufficient to check the first condition in the definition of $\lb_{\theta,r}$.
We will choose $r> 2\alpha/(\alpha-1) >2$, in which case
$\|\th_p\|_r < L^{\frac 2r}M$, where $L$ is the size of the torus. Now since $\alpha-1-2/r >0$ we have that 
\[
\|\th_p\|_r < L^{\frac 2r}M < c_{\alpha,r} \nu \lambda_p^{\alpha-1-\frac{2}{r}}
\]
for
\[
\lambda_p > \left(\frac{L^{\frac 2r}M}{c_{\alpha,r}\nu}\right)^{\frac{1}{\alpha-1-2/r}}=:N(r).
\]
If $\lb_{\theta} \geq N(r)$ then we are done since the first condition in the definition of $\lb_{\theta, r}$ is satisfied above $N(r)$. Otherwise, for $\lambda_p \in (\lb_{\theta}, N(r))$ we have
\[
L^{\frac{2}{r}-\frac{2}{r_0}}\lambda_p^{\frac{2}{r}-\frac{2}{r_0}} < N(r)^{\frac{2}{r}-\frac{2}{r_0}} \to 1 \qquad \text{as} \qquad r \to r_0^-, \qquad \mbox { with } \quad  r_0=\frac{4}{\alpha-1},
\]
and  
\[\lambda_p^{\frac{1-\alpha}2}\|\theta_p\|_{r_0}<{\textstyle \frac{1}{2}}c_{\alpha,r} \nu.\]
Therefore
\[
\begin{split}
\lambda_p^{1-\alpha+\frac{2}{r}} \|\theta_p\|_r &< {\textstyle \frac{3}{2}} \lambda_p^{1-\alpha+\frac{2}{r}}L^{2(\frac1r-\frac1{r_0})} \|\theta_p\|_{r_0} \\
&= {\textstyle \frac{3}{2}}\lambda_p^{\frac{1-\alpha}2}  \|\theta_p\|_{r_0} \lambda_p^{\frac{2}{r}-\frac{2}{r_0}}L^{2(\frac1r-\frac1{r_0})}\\
&<{\textstyle \frac{3}{4}}  c_{\alpha,r}\nu N(r)^{\frac{2}{r}-\frac{2}{r_0}}\\
&< c_{\alpha,r}\nu,
\end{split}
\]
provided $r_0-r$ is small enough, which means $\lambda_p > \lb_{\theta, r}$. This concludes the proof.
\end{proof}
 
\end{Lemma}

We will now estimate $\lb_\theta$ for $\theta \in B_{L^\infty}$.
To verify the first condition in the definition of $\lb_{\theta}$ we note that
\[
\begin{split}
\|\theta_p\|_{\frac{4}{\alpha-1}} &\leq \lambda_0^{\frac{1-\alpha}{2}}\|\theta_p\|_{\infty}
\leq \lambda_0^{\frac{1-\alpha}{2}}R_\infty 
 < {\textstyle \frac{1}{2}} \lambda_p^{\frac{\alpha-1}{2}} c_0\nu(\alpha-1)^2,
\end{split}
\]
provided $\lambda_p > \lambda_0^{-1}\left( \frac{2R_\infty L^\alpha}{(\alpha-1)^2c_0\nu} \right)^{\frac{2}{\alpha-1}}$.

For the second condition we estimate
\[
\sum_{p\leq q} \lambda_p \|\theta_{p}\|_{\infty} \leq 2\lambda_q R_\infty
< \lambda_q^{\alpha} c_0\nu (\alpha-1)^2,
\]
provided $\lambda_q> \left(\frac{2R_\infty}{(\alpha-1)^2c_0\nu}\right)^{\frac1{\alpha-1}}$. 

Therefore, 
\begin{equation} \label{eq:boundOnLambda}
\lb_{\theta}\leq\max\left\{\lambda_0^{-1}\left( \frac{2R_\infty}{(\alpha-1)^2c_0\nu} \right)^{\frac{2}{\alpha-1}}, \qquad \left(\frac{2R_\infty}{(\alpha-1)^2c_0\nu}\right)^{\frac1{\alpha-1}}\right\},
\end{equation}
which is finite since $\alpha>1$. Clearly, when the force is large enough, the first term in \eqref{eq:boundOnLambda} dominates and
\[
\lb_{\theta}\leq \lambda_0^{-1}\left( \frac{2R_\infty}{(\alpha-1)^2c_0\nu} \right)^{\frac{2}{\alpha-1}}.
\]
When $f \in L^2$, \eqref{eq:boundOnR}  gives
\[
R_\infty \sim \lambda_0^{1-\alpha}\left(\frac{\|f\|_2}{\nu}\right).
\]
Therefore
\[
\lb_{\theta}\lesssim L^{3}\left( \frac{\|f\|_2}{(\alpha-1)^2\nu^2} \right)^{\frac{2}{\alpha-1}}.
\]

\subsection{The critical case $\alpha=1$: proof of Theorem \ref{thm-theta-est2}}

In this section we consider the critical case $\alpha=1$ assuming that $f\in L^\infty\cap H^1$ and the initial data $\theta(0) \in H^1$. In this case it is known that there exists
a global solution $\theta \in C([0,\infty);H^1) \cap L^2_{loc}((0,\infty);H^{3/2})$. Moreover, 
\[
\|\theta(t)\|_{C^h} \lesssim \|\theta(t_0)\|_{\infty} + \frac{\|f\|_\infty}{\nu}, \qquad \forall t\geq t_0+\frac{3}{2(1-h)},
\] 
where $h =\min\left\{ \frac{\nu}{c_1\|\theta(0)\|_{\infty} + c_2\|f\|_\infty/\nu}, \frac{1}{4}\right\}$
(see Theorem 4.2 and (4.19) in  \cite{CCV}). We will not keep track of the length $L$ in this subsection in order not to overwhelm the estimates.
Since $f\in L^\infty$, the radius of $B_{L^\infty}$ in  \eqref{eq:boundOnR} becomes  $R_\infty \sim \frac{\|f\|_\infty}{\nu}$. So when $t$ is large enough we have  $\|\theta(t)\|_{C^h}  \leq R_{C^h} \sim \frac{\|f\|_\infty}{\nu}$.
Now we will estimate the determining wavenumber 
\[
\lb_{\theta,r}(t)=\min\{\lambda_q:\lambda_p^{\frac2r}\|\theta_{p}\|_r< c_{1,r}\nu \quad \forall p>q, \ \mbox { and }
\ \lambda_q^{-1}\sum_{p\leq q}\lambda_p\|\theta_p\|_\infty< c_{1,r}\nu   \},
\]
where  $r > 3$. Regarding the first condition, note that
\[
\|\theta_p\|_r \leq \|\theta_p\|_\infty  \leq \lambda_p^{-h}R_{C^h} < c_{1,r}\nu \lambda_p^{-\frac{2}{r}},
\]
provided
\[
\lambda_p \geq \left(\frac{R_{C^h}}{c_{1,r}\nu}\right)^{\frac{1}{h-2/r}}, \qquad \text{and} \qquad h> \frac{2}{r}.
\]
As for the second condition,
\[
\sum_{p\leq q} \lambda_p \|\theta_{p}\|_{\infty} \leq 2\lambda_q^{1-h} R_{C^h}
< c_{1,r}\nu \lambda_q
\]
provided
\[
\lambda_q > \left(\frac{2R_{C^h}}{c_{1,r}\nu}\right)^{\frac{1}{h}}.
\]

Therefore, 
\[
\lb_{\theta,r}\leq\max\left\{ \left(\frac{R_{C^h}}{c_{1,r}\nu}\right)^{\frac{1}{h-2/r}}, \qquad   \left(\frac{2R_{C^h}}{c_{1,r}\nu}\right)^{\frac{1}{h}}  \right\}.
\]
Since $h \lesssim \frac{\nu^2}{\|f\|_\infty}$ and $ R_{C^h} \sim \frac{\|f\|_\infty}{\nu}$, we obtain
\[
\lb_{\theta,r}\lesssim \left(\frac{\|f\|_\infty}{\nu^2}\right)^\frac{c\|f\|_\infty}{\nu^2},
\]
for some absolute (depending on $L$)  constant $c$ and some large enough $r$.


\section*{Acknowledgement}
This work was partially supported by NSF grants DMS--1108864 and DMS--1517583.
The authors thank anonymous referees for careful reading the manuscript and constructive comments.

\end{document}